\documentclass[a4paper,reqno]{amsart}
\usepackage[T1]{fontenc}
\usepackage[utf8x]{inputenc}
\usepackage[english]{babel}
\usepackage{amsmath,amsfonts,amssymb,upgreek,comment,mathtools}
\usepackage[11pt,curve,matrix,arrow,frame,graph]{xy}
\usepackage{graphicx}
\usepackage{hyperref}
\usepackage{enumerate}
\usepackage{xcolor}
\usepackage{esint}

\usepackage{amsthm}

\theoremstyle{plain}
\newtheorem*{theorem*}{Theorem}
\newtheorem{thmA}{Theorem}

\newtheorem{claim}{Claim}[section]
\newtheorem{theorem}{Theorem}[section]
\theoremstyle{definition}
\newtheorem{D}[theorem]{Definition}

\newtheorem{cor}[theorem]{Corollary}
\newtheorem{prop}[theorem]{Proposition}

\theoremstyle{definition}
\newtheorem{rem}[theorem]{Remark}
\newtheorem{rems}[theorem]{Remarks}
\newcommand{\R}{\ensuremath{\mathbb R}}
\newcommand{\N}{\ensuremath{\mathbb N}}
\newcommand{\Z}{\ensuremath{\mathbb Z}}

\newcommand{\bH}{\ensuremath{\mathbb H}}

\newcommand{\eps}{\ensuremath{\varepsilon}}

\newcommand{\Ric}{\ensuremath{\mbox{Ric}}}
\newcommand{\Ricm}{\ensuremath{\mbox{Ric}_{\mbox{\tiny{-}}}}}

\newcommand{\katoD}{\ensuremath{\mbox{k}_{D^2}(M^n,g)}}

\newcommand{\setR}{\mathbb{R}}

\newcommand{\bT}{\mathbb{T}}
\newcommand{\cD}{\mathcal{D}}\newcommand{\cF}{\mathcal{F}}

\newcommand{\Id}{\mathrm{Id}}

\newcommand{\di}{\mathop{}\!\mathrm{d}}

\newcommand{\dist}{\mathsf{d}}

\newcommand{\diam}{\mathrm{diam}}

\DeclareMathOperator{\RCD}{RCD}

\newfont{\tmpf}{cmsy10 scaled 2500}

\newcommand{\bGamma}{{\bf \Gamma}}

\def\R{\mathbb R}\def\N{\mathbb N}\def\bB{\mathbb B}
\def\cA{\mathcal A}

\def\cC{\mathcal C}
\def\cD{\mathcal D}

\def\cH{\mathcal H}

\newcommand{\df}{\coloneqq}
\newcommand{\fd}{\eqqcolon}

\title[Torus stability under Kato bounds on the Ricci curvature]{Torus stability under Kato bounds on the Ricci curvature}

\author{Gilles Carron}
\address{G. Carron, Nantes Université, CNRS, Laboratoire de Mathématiques Jean Leray, LMJL, UMR 6629, F-44000 Nantes, France.} 
\email{Gilles.Carron@univ-nantes.fr}

\author{Ilaria Mondello}
\address{I. Mondello, Université Paris Est Cr\'eteil, Laboratoire d'Analyse et Math\'ematiques appliqu\'es, UMR CNRS 8050, F-94010 Creteil, France.}
\email{ilaria.mondello@u-pec.fr}

\author{David Tewodrose}
\address{D. Tewodrose, Nantes Université, CNRS, Laboratoire de Mathématiques Jean Leray, LMJL, UMR 6629, F-44000 Nantes, France.}
\email{David.Tewodrose@univ-nantes.fr}
\date{}

\begin{document}

\maketitle

\begin{abstract}
We show two stability results for a closed Riemannian manifold whose Ricci curvature is small in the Kato sense and whose first Betti number is equal to the dimension. The first one is a geometric stability result stating that such a manifold is Gromov-Hausdorff close to a flat torus. The second one states that, under a stronger assumption,  such a manifold is diffeomorphic to a torus: this extends a result by Colding and Cheeger--Colding obtained in the context of a lower bound on the Ricci curvature.
\end{abstract}



\section{Introduction}
The celebrated Bochner theorem states that if a closed Riemannian manifold $(M^n,g)$ has non-negative Ricci curvature then its first Betti number satisfies
$$b_1(M)\le n,$$ with equality if and only if $(M^n,g)$ is isometric to a flat torus. This inequality was improved by Gromov \cite{Gromov1} and Gallot \cite{Gallot2} who found $\eps(n)>0$ such that if 
$$\Ric\ge -\frac{\eps(n)}{\mathrm{diam}(M,g)^2}\, g,$$then $b_1(M)\le n.$ Gromov also made a conjecture about the equality case, which was proven true by Colding \cite{Col97} and Cheeger--Colding \cite{ChCo97}: there exists $\updelta(n)>0$ such that if $(M^n,g)$ is a closed Riemannian manifold of diameter $D$ satisfying 
$$b_1(M)= n, \qquad  \Ric\ge -\frac{\updelta(n)}{D^2}\, g,$$
then $M$ is diffeomorphic to a torus.  The proof of this latter result consists in two steps. First, one shows that for any $\eps\in (0,1)$ there exists $\updelta(n,\eps)>0$ such that if $(M^n,g)$ with diameter $D$ satisfies
$$b_1(M)= n, \qquad  \Ric\ge -\frac{\updelta(n,\eps)}{D^2}\, g,$$
then there exists a flat torus $\bT^n$ and an $\eps D$-almost isometry\footnote{for any $\upeta>0$, a $\upeta$-almost isometry is a Borel map $\Phi$ between two metric spaces $(X,d_X)$ and $(Y,d_Y)$ such that: (i) $\left| d_X(x_0,x_1)- d_Y(\Phi(x_0),\Phi(x_1))\right|\le \upeta$ for any $x_0,x_1\in X$, (ii) for any $y \in Y$ there exists $x \in X$ such that $ d_Y(\phi(x),y)\le \upeta$} $\Phi\colon  M\rightarrow \bT^n$. Secondly,  the intrinsic Reifenberg theorem of Cheeger--Colding allows to prove a topological stability result: if $\eps$ is sufficiently small, then $M$ is diffeomorphic to $\bT^n$ (see \cite[Theorem A.1.1. and Theorem A.1.13]{CheegerColdingI}). For more details, we refer to the very instructive texts \cite{Gallot, ChPisa} presenting the work of Cheeger and Colding.

The Bochner estimate has been generalized in several directions.  For a Riemannian manifold $(M^n,g)$, let $\Ricm : M \to \setR_+$ be the lowest non-negative function such that for any $x \in M$,
\[
\Ric_x  \ge - \Ricm(x)g_x.
\]
Then Gallot obtained in  \cite{GallotInt}  that for every $p>n/2$, there exists $\eps(n,p)>0$ such that if $(M^n,g)$ with diameter $D$ satisfies
$$D^2\left(\fint_M \Ricm^p \di \nu_g\right)^{\frac1p}\le \eps(n,p),$$ then $b_1(M)\le n$; here and throughout, $\nu_g$ is the Riemannian volume measure induced by $g$ on $M$, and $\fint_A f \di \nu_g \coloneqq \nu_g(A)^{-1}\int_A f \di \nu_g$ for any Borel set $A \subset M$ and any measurable function $f$ defined on $A$.  To our knowledge,  no topological rigidity result has been obtained so far from this integral condition. It seems to us that the segment inequality proven in \cite{Chen:2020aa}  and the results from \cite{PW1,PW2} may imply such a  rigidity result. Another direction is the one of metric measure spaces satisfying a suitable synthetic Ricci curvature lower bound. In this context, a rigidity result \`a la Bochner and geometric stability results hold, see \cite{MR3814057,MR3987869,Mondello:2021aa}.

In this paper, we obtain geometric and topological results under a Kato bound. More precisely, let us introduce the following definition.

\begin{D} Let $(M^n,g)$ be a complete Riemannian manifold with heat kernel $H(t,x,y)$ and diameter $D$. For any $T>0$, we set
$$\mbox{k}_T(M^n,g):=\sup_{x\in M} \iint_{[0,T]\times M} H(t,x,y)\Ricm(y)\di\nu_g(y) \di t.$$
We say that  the number $\mbox{k}_{D^2}(M^n,g)$ is the Kato constant of $(M^n,g)$.
\end{D}

The first occurence of $\mbox{k}_T(M^n,g)$ in the study of Riemannian manifolds seems to be \cite{MR3412360}. The geometric and analytic consequences of a bound on $\mbox{k}_T(M^n,g)$ have been extensively studied since then, see e.g.~\cite{Rose,rose2018manifolds,C16,rosewei,CarronRose,braun2021heat,CMT,CMT2}. It is useful to note that if $\Ric\ge -\kappa^2 g$ then $\mbox{k}_T(M^n,g)\le \kappa^2 T$, hence a smallness assumption on $\mbox{k}_T(M^n,g)$ should be understood as a control on the part of the manifold where $\Ricm \gtrsim T^{-1}.$

The Bochner estimate extends to the case of Riemannian manifolds with small Kato constant,  as proven in \cite{Rose} and improved in \cite{C16}: there exists $\updelta(n)>0$ such that if $(M^n,g)$ is a closed Riemannian manifold of diameter $D$ such that $\mbox{k}_{D^2}(M^n,g)\le \updelta(n)$, then  $b_1(M)\le n.$ Our first main result provides an answer to a question raised in \cite{C16} about the equality case.
\begin{thmA}\label{Thm:A} For any $\eps\in (0,1)$ there exists $\updelta(n,\eps)>0$ such that if $(M^n,g)$ is a closed Riemannian manifold of diameter $D$ satisfying
$$b_1(M)= n \qquad \text{and} \qquad \mbox{k}_{D^2}(M^n,g)\le \updelta(n,\eps),$$
then $M$ is $\eps D$-almost isometric to a flat torus.
\end{thmA}

\begin{rem} From \cite[Theorem 4.3]{CarronRose} we can replace the smallness assumption $\mbox{k}_{D^2}(M^n,g)\le \updelta(n,\eps)$ with an integral condition involving the Ricci curvature only, namely
$$\sup_{x \in M}\int_0^Dr\fint_{B_r(x)}\Ricm(y)\di\nu_g(y)\,\di r\le\updelta(n,\eps).$$ 
\end{rem}

Our second main result provides a topological stability theorem under a so-called strong Kato bound.  This assumption appeared naturally in our previous work \cite{CMT,CMT2} where we obtained, among other results, Reifenberg regularity.

\begin{thmA}\label{Thm:B}Let $f\colon [0,1]\rightarrow \R_+$ be a non-decreasing function satisfying
\begin{equation}\label{eq:SK}\tag{SK}
\int_0^1 \frac{\sqrt{f(t)}}{t} \di t<+\infty.
\end{equation}
 Then there exists $\updelta(n,f)>0$ such that if a closed Riemannian manifold $(M^n,g)$ of diameter $D$ satisfies
\[
b_1(M)=n, \qquad \mbox{k}_{D^2}(M^n,g)\le \updelta(n,f), 
\] 
and
 \[
\mbox{k}_{tD^2}(M^n,g)\le f(t)  \quad \text{for all $t \in (0,1]$},
 \]
then $M$ is diffeomorphic to a torus.
\end{thmA}
\begin{rem} Our proof actually shows a stronger result: for any $\upalpha\in (0,1)$ there exists $\updelta(n,f,\upalpha)>0$ such that if $(M^n,g)$ is a closed Riemannian manifold of diameter $D$ satisfying $b_1(M)= n$,  $\mbox{k}_{D^2}(M^n,g)\le \updelta(n,f,\upalpha)$ and $\mbox{k}_{tD^2}(M^n,g)\le f(t)$ for any $t \in (0,1]$, with $f$ satisfying \eqref{eq:SK}, then there exist a flat torus $(\bT^n,\dist_{\bT^n})$ and a diffeomorphism
$\cA\colon M\rightarrow \bT^n$ such that for any $x,y\in M$,
$$\upalpha\left(\frac{\dist_g(x,y)}{D}\right)^{\frac 1\upalpha}\le \frac{\dist_{\bT^n}\left(\cA(x),\cA(y)\right)}{D}\le  \upalpha^{-1}\left(\frac{\dist_g(x,y)}{D}\right)^\upalpha.$$
Here and throughout, $\dist_g$ is the Riemannian distance induced by $g$. 
\end{rem}

According to \cite{RoseStollmann}, an $L^p$ smallness condition on $\Ricm$ yields the strong Kato bound, so that Theorem \ref{Thm:B} has the following corollary. 
\begin{cor}For any $p>n/2$, there exists $\eps(n,p)>0$ such that if $(M^n,g)$ is a closed Riemannian manifold of diameter $D$ satisfying
\[
b_1(M)= n \quad \text{and} \quad D^2\left(\fint_M \Ricm^p\di\nu_g\right)^{\frac1p}\le \eps(n,p),
\]
then $M$ is diffeomorphic to a torus. 
\end{cor}
Similarly,  using \cite[Theorem 4.3]{CarronRose}, we get the following corollary involving a suitable Morrey norm.
 \begin{cor}For any $\upalpha\in (0,2]$,  there exists $\eps(n,\upalpha)>0$ such that if $(M^n,g)$ is a closed Riemannian manifold of diameter $D$ satisfying
\[
b_1(M)= n \quad \text{and} \quad \sup_{\substack{x \in M\\r \in (0,D)}} D^{\alpha}r^{2-\alpha}\fint_{B_r(x)}\Ricm(y)\di\nu_g(y)\le \eps(n,\alpha),
\] 
then $M$ is diffeomorphic to a torus. 
\end{cor}

Colding's original argument relied upon harmonic approximations of Busemann-like functions.  Here we follow an alternative approach, closer to the one proposed by Gallot in \cite{Gallot}.  We consider the Albanese map $\cA\colon M\rightarrow \bT^n$. We lift $\cA$ to a harmonic map $\widehat \cA:=(\widehat \cA_1,\dots,\widehat \cA_n)$ defined on a suitable abelian cover $\widehat M{\longrightarrow} M$ which is equivariant under an action of $\Z^n$. Estimates from \cite{C16} imply that if $\mbox{k}_{D^2}(M^n,g)\le \updelta \le 1/(16n)$, then $\widehat \cA$ is surjective, $(1+C(n)\sqrt[3]{\updelta})$-Lipschitz and for any $r\in [D, \updelta^{-1/6}D]$,
\begin{equation}\label{eq:star}\tag{$\star_r$}
 \left(r^2\fint_{B_r} |\nabla d\widehat\cA|^2\di\nu_{\widehat g}\right)^{\frac 12} + \fint_{B_r} \left| \langle d\widehat \cA_i,d\widehat \cA_j\rangle -\updelta_{i,j}\right| \le C(n)\sqrt[3]{\updelta}.
\end{equation}

One original point in our proof of Theorem \ref{Thm:A} is the use of an almost rigidity result (Theorem \ref{thm:presEucl}) which implies that under such an estimate, the restriction of $\widehat \cA$ to a ball of radius $64n^2D$ realizes an $\eps D$-almost isometry with a Euclidean ball of the same radius.  We prove this almost rigidity result by means of the analysis developed in \cite{CMT}. Then we can follow the lines of Gallot's argument to get Theorem \ref{Thm:A}.

We prove Theorem \ref{Thm:B} by showing that $\cA$ is a diffeomorphism, answering a question raised by Gallot \cite[Section 6]{Gallot}. This  differs from Colding's proof which used the intrinsic Reifenberg theorem to conclude.  It is enough to show that the restriction of $\widehat{\cA}$ to the ball $B_D(\widehat{o})$ is a diffeomorphism onto its image.  In the context of almost non-negative Ricci curvature,  according to the recent Reifenberg theorem established in \cite[Theorem 7.10]{CJN}, this is the case if ($\star_{2D}$) holds and if the volume ratio $$\frac{\nu_{\widehat{g}}(B_{2D}(\widehat{o}))}{\omega_n (2D)^n}$$ is almost one.  In our context, we have at our disposal an analogous Reifenberg type result: see Proposition \ref{Reife}.  To apply this result, we must control a heat kernel ratio which plays the role, in our setting, of the volume ratio for Ricci curvature lower bounds. One main difference is that, unlike the volume, the heat kernel is a non-local quantity.  Our key tools to get the desired control are a heat kernel comparison theorem à la Cheeger--Yau \cite{CheegerYau}, and an almost Euclidean volume bound (Theorem \ref{thm=presquesurjective}, after \cite[Theorem 1.2]{cheeger_structure_2000}).

\hfill

\noindent \textbf{Acknowledgments:}
The authors thank the anonymous referee for constructive comments including Remark \ref{rk:ref_comment}.  They are partially supported by the ANR grant ANR-17-CE40-0034: CCEM. The first author is also partially supported by the ANR grant ANR-18-CE40-0012: RAGE.  The third author is supported by Laboratoire de Mathématiques Jean Leray via the project Centre Henri Lebesgue ANR-11-LABX-0020-01, and by Fédération de Recherche Mathématiques de Pays de Loire via the project Ambition Lebesgue Loire.

\section{The Dynkin condition and consequences}

In this section, we point out relevant properties of the so-called Dynkin condition. We say that a complete Riemannian manifold $(M^n,g)$ satisfies such a condition if there exists $T>0$ such that
\begin{align}\label{Dynkin}\tag{Dy}
 \mbox{k}_T(M^n,g)\le\frac{1}{16n} \, \cdot
 \end{align}

\subsection{Closed manifolds}

Let us first mention properties of closed Riemannian manifolds satisfying a Dynkin condition.

\subsubsection{Volume doubling} See \cite[Proposition 3.8]{C16} and \cite[Proposition 3.3]{CMT} for the next result.
\begin{prop}\label{VD} Let $(M^n,g)$ be a closed Riemannian manifold satisfying \eqref{Dynkin}.  Then there exists $C(n)>0$ such that for any $x\in M$ and $0<s\le r\le \sqrt{T}$,
$$  \frac{\nu_g\left(B_r(x)\right)}{\nu_g\left(B_s(x)\right)}\le C(n) \left(\frac{r}{s}\right)^{e^2 n} .$$
\end{prop}

\subsubsection{Heat kernel bounds} See \cite[Proposition 2.6]{CMT2} for the following.
\begin{prop}\label{Prop:heatK} Let $(M^n,g)$ be a closed Riemannian manifold satisfying \eqref{Dynkin}.   Then there exists $C(n)>0$ such that for any $x,y\in M$ and $t\in (0,T),$
\begin{enumerate}[i)]
\item 
$ \frac{C(n)^{-1}}{\nu_g(B_{\sqrt{t}}(x))}e^{-C(n)\frac{\dist_g^2(x,y)}{t}}\le H(t,x,y)\le \frac{C(n)}{\nu_g(B_{\sqrt{t}}(x))}e^{-\frac{\dist_g^2(x,y)}{5t}},$
\item $\left|d_x H(t,x,y)\right|\le  \frac{C(n)}{\sqrt{t}\nu_g(B_{\sqrt{t}}(x))}e^{-\frac{\dist_g^2(x,y)}{5t}}.$
\end{enumerate}
\end{prop}
When $T\ge \diam^2(M^n,g)$, we can get an estimate depending only on the volume of $M$.
\begin{prop}\label{harpheat} Let $(M^n,g)$ be a closed Riemannian manifold of diameter $D$ satisfying \eqref{Dynkin}.  If  $T\ge D^2$
then for any $x,y\in M$ and $t\in (D^2,T)$,
$$H(t,x,y)\le\left(1+C(n)\frac{D}{\sqrt{t}}\right)\frac{1}{\nu_g(M)} \, \cdot$$
\end{prop}
\proof Let $t\in (D^2,T)$ and $x,y\in M$. By stochastic completeness, we have that $\int_M H(t,z,y)\di\nu_g(z)=1$, hence there is some $z_0\in M$ such that $H(t,z_0,y)=\frac{1}{\nu_g(M)}$. By the previous proposition, we have that $\left|d_z H(t,z,y)\right|\le \frac{C(n)}{\sqrt{t}\nu_g(M)}.$  Then
$$\left| H(t,x,y)-H(t,z_0,y)\right|\le \dist_g(x,z_0)\,  \frac{C(n)}{\sqrt{t}\nu_g(M)}\le  \frac{C(n) D}{\sqrt{t}\nu_g(M)} \, \cdot$$
\endproof
\subsection{Non-compact manifolds}
For complete non-compact manifolds, the results of \cite[Subsection 3.1]{CMT} yield the following.
\begin{prop}\label{convergence} Let $(M^n,g,o)$ be a pointed complete  Riemannian manifold.  Assume that there exists a sequence $\left\{ (M_\alpha^n,g_\alpha,o_\alpha)\right\}_\alpha$ of pointed closed  Riemannian manifolds satisfying \eqref{Dynkin} for a same $T>0$. Assume that for any $R>0$ there exists $\alpha_R$ such that for any $\alpha \ge \alpha_R$ there exists a diffeomorphism onto its image
$$\Phi_\alpha\colon B_R(o)\rightarrow M_\alpha$$ such that the following convergence holds:
$$\lim_{\alpha\to+\infty} \left\|\Phi_\alpha^*g_\alpha-g\right\|_{\cC^0}=0.$$
Then $(M^n,g)$ satisfies \eqref{Dynkin}, the volume doubling estimate from Proposition \ref{VD} and the heat kernel estimates from Proposition \ref{Prop:heatK}.
\end{prop}

We recall that a group $\Gamma$ with neutral element $1$ is residually finite if and only if it admits a sequence of normal subgroups $\{\Gamma_j\}$ with finite index such that 
$$\bigcap_j \Gamma_j=\{1\}.$$
Then Proposition \ref{convergence} has the following useful application.

\begin{prop}\label{Prop:Katoforcovering} Let $\pi:\widehat{M}\rightarrow M$ be a normal covering of a closed Riemannian manifold $(M^n,g)$ with residually finite deck transformation group.
If $({M},g)$ satisfies \eqref{Dynkin}, then $(\widehat{M},\pi^*g)$ satisfies \eqref{Dynkin}, the volume doubling estimate from Proposition \ref{VD} and the heat kernel estimates from Proposition \ref{Prop:heatK}.
\end{prop}
\begin{proof}
We start by noticing that if $V\colon M\rightarrow \R$ is a bounded function then
$$e^{-t\Delta_{\pi^*g}}(V\circ \pi)=\left(e^{-t\Delta_{g}}V\right)\circ \pi.$$
Indeed, since $(M,g)$ is closed, $(\widehat{M}, \pi^*g)$ is stochastically complete and then we have uniqueness in $L^\infty$ of the solution of the heat equation with fixed initial condition
$$\left\{\begin{array}{l}
\left(\frac{\partial }{\partial t}+\Delta_{\pi^*g}\right)u=0,\\
u(0,\cdot)=V\circ \pi(\cdot).\end{array}\right. $$
But $e^{-t\Delta_{\pi^*g}}(V\circ \pi)$  and $\left(e^{-t\Delta_{g}}V\right)\circ \pi$ are both solutions of this Cauchy problem hence we get the desired equality.
Notice that $\Ricm(\pi^*g)=\Ricm(g)\circ\pi$, hence for any $x\in \widehat{M}$ and $t>0$
$$\int_{\widehat{M}} H_{\pi^*g}(t,x,y) \Ricm(\pi^*g)(y) \di\nu_{\pi^*g}(y)=\int_M H_g(t,\pi(x),z) \Ricm(g)(z) \di\nu_{g}(z),$$
where we have noted $H_{\pi^*g}$ (resp. $H_g$) the heat kernel of $(\widehat{M}, \pi^*g)$ (resp. of $({M}, g)$). Hence we get that for any $T>0$:
\begin{equation}\label{eq:katochapeau}
\mbox{k}_T(M,g)=\mbox{k}_T(\widehat{M}, \pi^*g).
\end{equation}
As $\Gamma$ is residually finite,  there exists a sequence of normal subgroup $\Gamma_j\triangleleft\Gamma$ of finite index such that 
$$\bigcap_j \Gamma_j=\{1\}.$$
 For any $j\in \N^*$, we set 
  $\widehat M_j:=\widehat M/\Gamma_j.$
We get two covering maps
$$\widehat M\stackrel{p_j}{\longrightarrow}\widehat M_j\stackrel{\pi_j}{\longrightarrow}M.$$
Note that $\widehat M_j$ is a closed manifold and the above argument implies that 
 for any $T>0$:
$$\mbox{k}_T(\widehat{M}, \pi^*g)=\mbox{k}_T(\widehat{M}_j, \pi_j^*g).$$
Hence each $(\widehat{M}_j, \pi_j^*g)$ satisfies \eqref{Dynkin}.  Moreover,  if we consider $\widehat{o}\in \widehat{M},$ then for any $R>0$ there is some $j_R$ such that for any $j\ge j_R$, the restriction of the covering map 
$$p_j\colon B_R(\widehat{o})\rightarrow B_R(p_j(\widehat{o}))$$ is an isometry. Then the conclusion follows from Proposition \ref{convergence}.
\end{proof}

\section{Almost rigidity results}

In this section, we provide almost rigidity results which are consequences of our previous work \cite{CMT,CMT2}. For any $\rho>0$, we let $\bB^n_\rho$ be the Euclidean ball in $\setR^n$ centered at $0$ with radius $\rho$. If $B$ is a ball in an $n$-dimensional Riemannian manifold and $h = (h_1,\ldots,h_n) :B \to \setR^n$ is  a smooth map, we denote by
\[
dh \, ^tdh  = \left[ \langle dh_i,dh_j \rangle \right]_{1\le i,j \le n}
\]
its Gram matrix map.  We write $\mathrm{Id}_n$ for the identity matrix of size $n$ and $\omega_n$ for the Lebesgue measure of the unit Euclidean ball in $\setR^n$. 

\subsection{Harmonic almost splitting}

\begin{theorem}\label{thm:presEucl} For any $\eps\in (0,1)$, there exists $\updelta=\updelta(n,\eps)>0$ such that for any closed Riemannian manifold  $(M^n,g)$ satisfying $\mbox{k}_{\rho^2}(M^n,g)\le \updelta$ for some $\rho >0$,  if for some $x \in M$ there exists a harmonic map
$$h\colon B_{\updelta^{-1}\rho}(x)\rightarrow \R^n$$ such that for any $r\in [\rho,\updelta^{-1}\rho)$,
$$ \left(r^2\fint_{B_r(x)} |\nabla dh|^2\di\nu_g\right)^{\frac 12}+\fint_{B_{r}(x)}\left| dh\,{}^tdh-\Id_n\right| \di\nu_g\le \updelta,$$
then $h$ is an $\eps\rho$-almost isometry between $B_\rho(x)$ and $\bB^n_\rho$.
\end{theorem}

\begin{rems}\label{cover} \begin{enumerate}[i)]
\item From the proof of Proposition \ref{Prop:Katoforcovering}, we easily see that the result also holds if we assume that $(M^n,g)$ is a normal covering of a closed Riemannian manifold whose deck transformation group is residually finite.
\item A similar statement holds for $\RCD(K,N)$ spaces \cite[Proposition 3.7]{MR4277822}. 
\end{enumerate}
\end{rems}
\proof By scaling,  there is no loss of generality in assuming $\rho=1$, what we do from now on.  We argue by contradiction.   Assume that there exists $\eps\in (0,1)$ and:
\begin{itemize}
\item a sequence of positive numbers $\{\updelta_\alpha\}$ such that $\updelta_\alpha \downarrow 0$,
\item a sequence of pointed closed Riemannian manifolds $\{(M^n_\alpha,g_\alpha,x_\alpha)\}$ such that $\mbox{k}_{1}(M_\alpha,g_\alpha)\le \updelta_\alpha$ for any $\alpha$,
\item a sequence of maps $\{h_\alpha\colon B_{\updelta_\alpha^{-1}}(x_\alpha)\rightarrow \R^n\}$ such that for any $\alpha$, 
\begin{equation}\label{eq:hyp}
\text{$h_\alpha$ is not an $\eps$-almost isometry between $B_1(x_\alpha)$ and $\bB_1^n$}
 \end{equation}
and for any $r\in \left[1,\updelta_\alpha^{-1}\right],$
$$ \left(r^2\fint_{B_r(x_\alpha)} |\nabla dh_\alpha|^2\di\nu_{g_\alpha}\right)^{\frac 12}+\fint_{B_{r}(x_\alpha)}\left| dh_\alpha\,{}^tdh_\alpha-\Id_n\right| \di\nu_{g_\alpha}\le \updelta_\alpha.$$
\end{itemize}
We can assume the following. 
\begin{enumerate}
\item Thanks to \cite[Corollary 2.5, Remark 4.9]{CMT}, the sequence $\{(M^n_\alpha,g_\alpha,x_\alpha)\}$ converges in the pointed measured Gromov-Hausdorff topology to a space $(X,\dist,\mu,x)$ which is infinitesimally Hilbertian in the sense of \cite{GigliMAMS}. This limit space is  endowed with a carré du champ $\Gamma$ and a natural Laplace operator $L$, that is the Friedrichs extension of the quadratic form
$\varphi\in W^{1,2}\mapsto \int_X \di\Gamma(\varphi,\varphi)$. For the precise definitions of $\Gamma$ and $L$, see \cite[Section 1.2]{CMT} and references therein. 
\item By \cite[Proposition E.10]{CMT}, the sequence  $\{h_\alpha\}_\alpha$ converges uniformly on compact sets to a harmonic function
$h=(h_1,\dots,h_n)\colon X\rightarrow \R^n$ and for any $i,j\in \{1,\dots,n\}$,
$$\frac{\di\Gamma}{\di \mu}(h_i,h_j)=\updelta_{i,j}  \qquad \text{$\mu$-a.e.~on $X$.}$$
\end{enumerate}
Moreover,  we know that $(X,\dist,\mu)$ admits a locally Lipschitz heat kernel $H:(0,+\infty)\times X\times X \to (0,+\infty)$ which satisfies the following Li-Yau inequality \cite[Proposition 2.9 and Remark 2.10]{CMT2}:
for any $x\in X$, $t>0$, and $\mu$-a.e.~$y\in X$,
$$|d_y H(t,x,y)|^2- H(t,x,y)\frac{\partial H}{\partial t}(t,x,y)\le \frac{n}{2t} H^2(t,x,y).$$
Let $U\colon (0,+\infty)\times X\times X\rightarrow \R$ be such that for any $x,y \in X$ and $t>0$,
$$H(t,x,y)=\frac{e^{-\frac{U(t,x,y)}{4t}}}{(4\pi t)^{\frac n2}} \, \cdot$$ It is easy to check (see  \cite[Formula (83)]{CMT}) that $U$ satisfies, for any $(x,t)\in X \times (0,+\infty)$,
$$L U(t,x,\cdot)\ge -2n$$
in a weak sense, that is to say,  $v(\cdot)=U(t,x,\cdot)\in W^{1,2}_{loc}$ and 
for any non-negative $\varphi\in W^{1,2}_{c}$:
$$\int_X \di \Gamma(\varphi,v)\ge -2n\int_X\varphi\di\mu.$$
According to Varadhan formula (\cite[Proposition 1.6]{CMT}), for any $x,y\in X$,  $$\lim_{t\to 0} U(t,x,y)=\dist^2(x,y).$$
For any $x,y \in X$, set
\[
\rho(x,y)=|h(x)-h(y)|^2
\]
and note that
$$L \rho(x,\cdot)=-2n,$$
hence for any $t>0$ the function $\rho(x,\cdot)-U(t,x,\cdot)$ is sub-harmonic; passing to the limit $t \downarrow 0$ we get that $\rho(x,\cdot)-\dist^2(x,\cdot)$ is subharmonic.  Moreover, for each $\xi=(\xi_1,\dots,\xi_n)$, the function $h_\xi:=\sum_i \xi_ih_i$ satisfies $|dh_\xi|^2=|\xi|^2 $, so that $h_\xi$ is $|\xi|$-Lipschitz. As a consequence, for any $x,y \in X$,
\[
\rho(x,y) = \sup_{\substack{\xi \in \setR^n\\|\xi|=1}} |\langle \xi, h(x)-h(y) \rangle|^2 = \sup_{\substack{\xi \in \setR^n\\|\xi|=1}} |h_\xi(x)-h_\xi(y)|^2 \le \dist^2(x,y).
\]
 Therefore, the function $\rho(x,\cdot)-\dist^2(x,\cdot)$ is sub-harmonic, non-positive and it reaches its maximum value, zero, at $y=x$, hence it is constantly equal to zero.  Thus we have shown that $h\colon X\rightarrow \R^n$ is an isometry onto its image. But $h(X)$ is closed, convex and according to \cite[Claim 4.1,page 130]{CT19}, its convex hull is $\R^n$.  Then $h$ is an isometry between $(X,\dist)$ and the Euclidean space $\R^n$. By uniform convergence, we get that for $\alpha$ large enough, $h_\alpha\colon B_1(x_\alpha)\rightarrow \R^n$ is an $\eps$-almost isometry between $B_1(x_\alpha)$ and an Euclidean ball of radius $1$; this contradicts \eqref{eq:hyp}.
\endproof

\begin{rem}\label{rk:ref_comment} Our proof avoids the use of the $\RCD$ theory, however a better result could be proven using a recent result by Brué-Naber-Semola \cite{BNS}:\emph{ There is some constant $c(n)\in (0,1)$ so that for any $\eps\in (0,1)$, there exists $\updelta=\updelta(n,\eps)>0$ such that for any closed Riemannian manifold  $(M^n,g)$ satisfying $\mbox{k}_{\rho^2}(M^n,g)\le \updelta$ for some $\rho >0$,  if for some $p \in M$ there exists a harmonic map
$$h\colon B_{c(n)\rho}(p)\rightarrow \R^n$$ satisfying:
$$ \left(\rho^2\fint_{B_{\rho}(p)} |\nabla dh|^2\di\nu_g\right)^{\frac 12}+\fint_{B_{\rho}(p)}\left| dh\,{}^tdh-\Id_n\right| \di\nu_g\le \updelta,$$
then $h$ is an $\eps\rho$-almost isometry between $B_{c(n)\rho}(p)$ and $\bB^n_{c(n)\rho}$.}

This improvement would be the consequence of \cite[Theorem 3.8 and Remark 3.10]{CMT2}, of \cite[Theorem 3.8]{BNS} and of the  following corollary of  \cite[Theorem 4.11]{CMT}:
\emph{ for every $\upeta\in (0,1)$ there is some $\updelta_1 = \updelta_1(n,\eta)>0$ such that if $(M^n,g)$ is a closed Riemannian manifold  $(M^n,g)$ satisfying $\mbox{k}_{\rho^2}(M^n,g)\le \updelta_1$ and $p\in M$, then there is a pointed $\RCD(0,n)$ space  $(X,\dist_X,\mu_X,x)$ (which depends on $p$) such that when $M$ is endowed with the geodesic distance and with the measure $\mu_M=\nu_g/\nu_g\left(B_\rho(p)\right)$, then
$$\dist_{mGH}\left( B_\rho(p), B_\rho(x)\right)\le \upeta \rho.$$
}
\end{rem}

\subsection{Reifenberg regularity result}

Let $(M^n,g)$ be a complete Riemannian manifold. For any $(t,x) \in \setR_+ \times M$, we set
$$ \uptheta(t,x):=(4\pi t)^{\frac n2} H(t,x,x).$$ 
The quantity $\uptheta$ is an on-diagonal heat kernel ratio, as $(4\pi t)^{-\frac{n}{2}}$ is the on-diagonal Euclidean heat kernel. In \cite{CMT,CMT2},  we showed that if $(M^n,g)$ is closed and satisfies a strong Kato bound, then the quantity $\uptheta$ is almost monotone and controls the geometry of $M$.  In this regard, we have at our disposal the following Reifenberg regularity result, which is a consequence of \cite[Theorem 5.19]{CMT2} and the proof of Proposition \ref{Prop:Katoforcovering}.  The case of almost non-negative Ricci curvature was originally proven in \cite[Theorem 7.10]{CJN}.

\begin{prop}\label{Reife} 
Let $f\colon (0,1]\rightarrow \R_+$ be a non-decreasing function satisfying \eqref{eq:SK}.  Then there exists $\upbeta(n,f)>0$ such that for any complete Riemannian manifold $(M^n,g)$ which is a normal covering of a closed Riemannian manifold with residually finite deck transformation, if there exist $x \in M$, $R>0$ and $h:B_{R}(x)\to \setR^n$ harmonic such that:
\begin{itemize}
\item[(1)] $\mbox{k}_{R^2}(M^n,g)\le \upbeta(n,f)$ and $\mbox{k}_{tR^2}(M^n,g)\le f(t)$ for any $t \in (0,1]$,
\item[(2)] $\uptheta(R^2,x)\le 1+\upbeta(n,f)$,
\item[(3)] $\|dh\|_{L^\infty}\le 2$, $h(x)=0$ and 
$$ \left(R^2\fint_{B_R(x)} |\nabla dh|^2\di\nu_{g}\right)^{\frac 12}+\fint_{B_R(x)}\left| dh\,{}^tdh-\Id_n\right| \di\nu_{g}\le \upbeta(n,f),$$
\end{itemize}
then the restriction of $h$ to $B_{3R/4}(x)$ is a diffeomorphism onto its image.\end{prop}

\section{Albanese Maps}\label{section:albanese}

In this section, we recall the construction of the Albanese maps and derive some relevant properties.

\subsection{Construction of the Albanese maps} Let $(M^n,g)$ be a closed Riemannian manifold. We write $H_1(M,\Z)$ for the first integer-valued homology group of $M$, $H^1(M,\R)$ for its first real-valued cohomology group, $H^1_{dR}(M)$ for its first De Rham cohomology space, and  \[
 \cH^1(M^n,g) \df \left\{\alpha\in \cC^\infty(T^*M) \, : \, d\alpha=d^*\alpha=0\right\}
 \]
for the space of harmonic $1$-forms of $(M^n,g)$.  Let $b_1$ be the first Betti number of $M$.  Then the torsion free part $\Gamma$ of $H_1(M,\mathbb{Z})$ is isomorphic to $\mathbb{Z}^{b_1}$ (hence it is a residually finite group) and satisfies
$$\Gamma= \pi_1(M)/\Lambda$$
where
$$\Lambda\df\left\{\gamma\in \pi_1(M) \, : \,  \int_\gamma\alpha=0 \quad \text{for all $[\alpha]\in H^1(M,\R)$}\right\}.$$
Moreover, $\Gamma$ is the deck transformation group of the normal covering
\[
 \widetilde M/\Lambda \fd \widehat M \stackrel{\pi}{\longrightarrow} M
\]
where $\widetilde M$ is the universal cover of $M$. 

By the Hodge--de Rham theorem,  there exists a normalized $L^2$-orthonormal family of harmonic $1$-forms $\alpha_1,\dots,\alpha_{b_1}$, i.e.
\[
\fint_M \langle \alpha_i,\alpha_j \rangle = \delta_{i,j}
\]
for any $i,j$,  such that $[\alpha_1],\ldots,[\alpha_{b_1}]$ form a basis of $H^1_{dR}(M)$.  We choose $o\in M$ and $\widehat o\in \widehat M$ such that $\pi(\widehat o)=o$.  Then for any $i\in \{1,\ldots,b_1\}$ there exists a unique harmonic function ${\widehat \cA}_i : \widehat M \to \setR$ such that
 $$d\widehat \cA_i=\pi^*\alpha_i \quad \text{ and } \quad \widehat \cA_i\left(\widehat{o}\right)=0.$$ 
This yields a harmonic map 
\begin{equation}\label{defAtop}
\widehat{\cA} := (\widehat \cA_1,\ldots,\widehat \cA_{b_1}) \colon \widehat M\rightarrow \R^{b_1}.
\end{equation}

Let us now consider the linear map
\begin{equation}\label{eq:rho}
 \begin{array}{ccccl}
\rho &  : & \Gamma  &  \to & \setR^{b_1}\\
& & \gamma & \mapsto & \left(\int_\gamma\alpha_1,\dots, \int_\gamma\alpha_{b_1}\right)
\end{array}
\end{equation}
and set
\[
 \bGamma \df \rho(\Gamma).
\]
Then $\rho\colon \Gamma\rightarrow \bGamma$ is an isomorphism and $\bGamma$
 is a lattice of $\R^{b_1}$.  We endow $\R^{b_1} / \bGamma$ with the flat quotient Riemannian metric $g_{\R^{b_1} / \bGamma}$.  Note that $\widehat{\cA}$ is $\bGamma$-equivariant, that is, for any $\gamma\in \Gamma$ and $\widehat{x} \in \widehat{M}$,
\begin{equation}\label{eq:equivariance}
 \widehat{\cA}(\gamma.\widehat{x})=\widehat{\cA}(\widehat{x})+\rho(\gamma).
\end{equation}
Then $\widehat{\mathcal{A}}$ induces a harmonic quotient map
\begin{equation}\label{eq:rho}
 \begin{array}{ccccl}
\cA &  : & M=\widehat M/\Gamma  &  \to & \R^{b_1} / \bGamma\\
& & \Gamma.\widehat{x}  & \mapsto & \widehat{\cA}(\widehat{x})+\bGamma.
\end{array}
\end{equation}
We say that $\mathcal{A}$ is the Albanese map of $M$ and $\widehat{\mathcal{A}}$ is the lifted Albanese map of $M$. Note that by construction, the following diagram commutes:
$$\xymatrix{
{\widehat M}\ar[d]^{\pi}\ar[r]^{\widehat \cA} &{ \  \ \R^{b_1}}\ar[d]^{p}\\
{M}\ar[r]^{\cA }&{\R^{b_1}\!/\bGamma} .
} $$

 \subsection{Some estimates for harmonic $1$-forms}
In the next proposition, we derive some estimates for the elements of $\cH^1(M^n,g)$ under a smallness assumption on the Kato constant.

 \begin{prop}\label{prop:estiharm}  Let $(M^n,g)$ be a closed Riemannian manifold of diameter $D$ such that $\mbox{k}_{D^2}(M^n,g)\le \updelta$ for some $\updelta \in (0,1/(16n))$. Then there exists $C(n)>0$ such that for any $\alpha\in \cH^1(M^n,g)$,
 \begin{enumerate}[i)]
 \item $\displaystyle \|\alpha\|^2_{L^\infty}\le \left(1+C(n)\sqrt[3]{\updelta\,}\,\right)\fint_M |\alpha|^2.$
 \item $\displaystyle \fint_M \left| \, |\alpha|^2- \fint_M |\alpha|^2\, \right|\le C(n)\sqrt[3]{\updelta\,}\,\fint_M |\alpha|^2.$
 \item $\displaystyle D^2\fint_M |\nabla \alpha|^2\le C(n)\updelta\,\fint_M |\alpha|^2.$
 \end{enumerate}
 \end{prop}
 \begin{proof}   Set $N:=\left\lfloor \frac{1}{\updelta 16n}\right\rfloor$.  Proceeding like in \cite[Lemma 2.22]{C16}, for instance, we get that for any $\ell\in  \{1,\dots, N\}$, 
 \begin{equation}\label{dyngrandT}\mbox{k}_{\ell D^2}(M^n,g)\le\ell \, \mbox{k}_{D^2}(M^n,g) \le \ell \updelta\le   \frac{1}{16n} \, \cdot\end{equation}
 
 We first prove i). Let $\alpha\in \cH^1(M^n,g)$. By the Bochner formula,
$$|\nabla\alpha|^2+\frac12 \Delta |\alpha|^2+\Ric(\alpha,\alpha)=0.$$
We fix $x \in M$. We multiply the previous identity evaluated in $y \in M$ by the heat kernel $H(t,x,y)$ and integrate with respect to $(t,y)\in [0,\ell D^2]\times M$, where $\ell\in \{1,\dots, N\}$ is suitably chosen later.  This gives
\begin{align*}
& -  \iint_{[0, \ell D^2]\times M} \! \Ric(\alpha,\alpha)(y) \Delta_y H(t,x,y)\di \nu_g(y) \di t\\
& = \iint_{[0, \ell D^2]\times M}  \!  |\nabla \alpha|^2(y)H(t,x,y)  \di \nu_g(y) \di t \\
& + \frac{1}{2} \iint_{[0, \ell D^2]\times M} \! |\alpha|^2(y) \Delta_y H(t,x,y)\di \nu_g(y) \di t
\end{align*}
Since $\Ric \ge -\Ricm g$ and $\Delta_y H (t,x,y)= \frac{\partial}{\partial t} H(t,x,y)$,  we get
\begin{align*}
 |\alpha|^2(x)&\le \int_M H(\ell D^2,x,y) |\alpha|^2(y)\di\nu_g(y)\\
 & +2\|\alpha\|^2_{L^\infty}\iint_{[0,\ell D^2]\times M} \! H(t,x,y)\Ricm(y)\di\nu_g(y)\di t\\
 &\le\left(1+C(n)\frac{1}{\sqrt{\ell}}\right)\fint_M |\alpha|^2+2\mbox{k}_{\ell D^2}(M^n,g)\,\|\alpha\|^2_{L^\infty} \quad \text{by  Proposition \ref{harpheat}}\\
 &\le\left(1+C(n)\frac{1}{\sqrt{\ell}}\right)\fint_M |\alpha|^2+2\ell\updelta\,\|\alpha\|^2_{L^\infty} \qquad \qquad \, \,\quad \text{by \eqref{dyngrandT}. }
 \end{align*}
Thus
 $$\left(1-2\ell\updelta\right)\,\|\alpha\|^2_{L^\infty}\le \left(1+C(n)\frac{1}{\sqrt{\ell}}\right)\fint_M |\alpha|^2.$$
Choosing $\ell$ of the same order as $\updelta^{-\frac 23}$ yields the desired estimate. \par 
Let us now prove ii).  Consider $\alpha \in \cH^1(M^n,g)$. Up to rescaling, we may assume that
\[
\fint_M |\alpha|^2 = 1.
\]
Then
 \begin{align*}\fint_M\left| |\alpha|^2-1\right|&=\frac{2}{\nu_g(M)}\int_{\{|\alpha|^2\ge 1\}} \left( |\alpha|^2-1\right)\\
 & \le \frac{2\nu_g\left(\{|\alpha|^2\ge 1\}\right)}{\nu_g(M)}C(n)\sqrt[3]{\updelta\,}\\
 & \le 2C(n)\sqrt[3]{\updelta\,},
 \end{align*}
 where we have used that (i) and then $\nu_g\left(\{|\alpha|^2\ge 1\}\right)\le \int_M |\alpha|^2.$
 \par Let us prove iii). We integrate the Bochner formula over $M$, divide by  $\nu_g(M)$, and use i) to get
 $$\fint_M |\nabla \alpha|^2=-\fint_M \Ric(\alpha,\alpha)\le C(n)\fint_M \Ricm\    \fint_M |\alpha|^2.$$
But 
 \begin{align*}
 D^2\fint_M \Ricm&=\frac{1}{\nu_g(M)}\int_{[0,D^2]\times M\times M} H(t,x,y)\Ricm(y)\di\nu_g(y)\di\nu_g(x)\di t\\
 & \le \mbox{k}_{D^2}(M^n,g) \le \updelta.
 \end{align*}
\end{proof}

\begin{rem}
By the Grothendieck Lemma \cite[Theorem 5.1]{Rudin} (see also \cite{Li} and \cite[Théorème 4]{GallotMeyer}), the previous proposition implies that for any  closed Riemannian manifold $(M^n,g)$ of diameter $D$ such that $\mbox{k}_{D^2}(M^n,g)\le \updelta$ for some $\updelta \in (0,1/(16n))$,
$$b_1(M)\le \left(1+C(n)\sqrt[3]{\updelta\,}\,\right)\, n,$$ so that:
$$\updelta<\updelta_n:=\left(n C(n)\right)^{-3} \qquad \Rightarrow \qquad b_1(M)\le n.$$
In particular, this provides another proof of \cite[Proposition 3.12]{C16}.
\end{rem}

\subsection{Consequences for the Albanese maps}
 
The previous estimate imply the following.
 
\begin{prop}\label{propAlbanese} There exists $\updelta(n)>0$ such that if $(M^n,g)$ is a closed Riemannian manifold of diameter $D$ satisfying $b_1(M)=n$ and $\mbox{k}_{D^2}(M^n,g)\le \updelta$ for some $\updelta \in (0,\updelta(n)]$, then the Albanese maps $\cA$ and $\widehat\cA$ satisfy the following properties.
\begin{enumerate}[i)]
\item \label{LipA} They are  $(1+C(n)\sqrt[3]{\updelta\,})$-Lipschitz. 
\item \label{degA} They are surjective.
\item \label{iii}  For any $r\in [D,\updelta^{-1/6}D]$,
\begin{equation}\label{spliA}
 \left(r^2\fint_{B_r(\widehat o)} |\nabla d\widehat\cA|^2\di\nu_{\widehat g}\right)^{\frac 12}+\fint_{B_{r}(\widehat o)}\left| d\widehat\cA\,{}^td\widehat\cA-\Id_n\right| \di\nu_{\widehat g}\le C(n)\sqrt[3]{\updelta}
 \end{equation}
and 
  \begin{equation}\label{volAhatM}
  \nu_{\widehat g}(B_r(\widehat o))\ge \left( 1-C(n)\sqrt[3]{\updelta}\right) \omega_n r^n.
\end{equation}
\end{enumerate}
\end{prop}

\proof[Proof] Let $\pi$ be the projection map from $\widehat M$ to $M$, and let $(\alpha_1,\ldots,\alpha_n)$ be the orthonormal basis of harmonic $1$-forms used to build $\widehat \cA$.  We let $C(n)>0$ be a generic constant depending only on $n$ whose value may change from line to line.

Let us prove \ref{LipA}).  For $\xi=(\xi_1,\dots,\xi_n)\in \R^n$, set $ \alpha:=\sum_i\xi_i\alpha_i \in \cH^1(M^n,g)$ and observe that 
 $$|\xi|^2 = \fint_M |\alpha|^2.$$
Then for any $ x\in \widehat M$ and $ v\in T_{x} \widehat M$, since $d \widehat \cA = (\pi^*\alpha_1,\ldots,\pi^*\alpha_n)$,
 $$ \left|\langle \xi,d_{x}\widehat\cA({v})\rangle\right|=\left|\alpha(\pi( x))(d_{x}\pi(v))\right|\le |\xi| \left(1+C(n)\sqrt[3]{\updelta\,}\,\right)^{\frac12}\sqrt{\widehat g_{x}(v,v)},$$
where we have applied i) in Proposition \ref{prop:estiharm}. This yields \ref{LipA}).

Let us now prove \ref{degA}).  Proposition \ref{prop:estiharm} implies that $\Omega:=\alpha_1\wedge\dots\wedge\alpha_n$ satisfies
 $$\fint_M \left| |\Omega|-1\right|\di\nu_g\le C(n)\sqrt[3]{\updelta\,}.$$
 Hence if $M$ is oriented and up to permutation of $\alpha_1$ and $\alpha_2$ then 
 $$\left|\int_M \Omega-\nu_g(M)\right|\le C(n)\sqrt[3]{\updelta\,} \nu_g(M),$$
 but by construction, setting  ${\bf\Omega}:=dx_1\wedge\dots\wedge dx_n$, we have $\Omega=\cA^{*}{\bf\Omega}$, hence 
 $$ \left| \deg \cA-\frac{\nu_g(M)}{\text{vol}\, \R^n/\bGamma} \right| = \frac{1}{\text{vol}\, \R^n/\bGamma} \left|\int_M\cA^{*}{\bf\Omega}-\nu_g(M)\right|\le C(n)\sqrt[3]{\updelta\,}\frac{\nu_g(M)}{\text{vol}\, \R^n/\bGamma} \, \cdot$$
 Hence if $C(n)\sqrt[3]{\updelta\,}<1$, then $\deg \cA\not=0$ and $\cA$ is surjective.  If $M$ is not oriented, then using the two-fold oriented cover $M_o\stackrel{\pi_o}{\longrightarrow} M$, the same argumentation can be applied to $\pi_o^*\Omega$ in order to get that $\cA\circ \pi_o\colon M_o\rightarrow \R^n/\bGamma$ is surjective.

Let us now prove \ref{iii}).  Observe that \eqref{volAhatM} is a direct consequence of \eqref{spliA} and \cite[Theorem 1.2]{cheeger_structure_2000} (see also Theorem \ref{thm=presquesurjective}). Thus we are left with proving \eqref{spliA}. To this aim,  we use the following result: for any $f \in L^1(M)$ and $r\in [D,\sqrt{N}D]$,
\begin{equation}\label{eq:claim}
\fint_{B_r(\widehat o)}\left|f\circ\pi\right|\le C(n) \fint_M |f|.
\end{equation}
Together with ii) in Proposition \ref{prop:estiharm}, this implies that for any $r\in [D,\sqrt{N}D],$
$$\fint_{B_r(\widehat o)} \left| d\widehat\cA \, ^td\widehat\cA-\Id_n\right| \le C(n)\sqrt[3]{\updelta},$$
and similarly iii) in Proposition \ref{prop:estiharm} yields that if $r\in [D,\updelta^{-\frac16}D],$ then
$$r^2\fint_{B_r(\widehat o)} \left| \nabla d\widehat\cA\right|^2\le C(n) \updelta \left(\frac{r}{D}\right)^2\le C(n)\updelta^{\frac23}.$$
Thus we are left with proving \eqref{eq:claim}. Let $\cD\subset B_D(\widehat o)$ be a fundamental domain for $\Gamma \longrightarrow \widehat M\stackrel{\pi}{\longrightarrow} M.$ Set $$G(r)=\left\{\gamma \in \Gamma : \gamma \cD\cap B_r(\widehat o)\not=\emptyset\right\}.$$
Then
$$B_r(\widehat o)\subset \bigcup_{\gamma \in G(r)} \gamma \cD\subset B_{r+D}(\widehat o),$$
so that
$$\nu_{\hat{g}}(B_r(\widehat o))\le \#G(r)\nu_{\hat{g}}(\cD)=\#G(r)\nu_g(M)\le \nu_{\hat{g}}(B_{r+D}(\widehat o)).$$
The group $\Gamma\simeq \Z^n$ is residually finite, hence 
Proposition \ref{Prop:Katoforcovering} and Proposition \ref{VD}
imply the volume doubling estimate: for any $0<r\le R\le \sqrt{N}D$,
$$ \nu_{\hat{g}}(B_R(\widehat o))\le C(n) \left(\frac{R}{r}\right)^{e^2n} \nu_{\hat{g}}(B_r(\widehat o)).$$
If $r \ge D$ and $r + D \le \sqrt{N} D$, this yields
$$\#G(r)\nu_g(M)\le C(n)\left(\frac{r+D}{r}\right)^{e^2n} \nu_{\hat{g}}(B_r(\widehat o))\le C(n)\nu_{\hat{g}}(B_r(\widehat o)),$$
so that 
\begin{align*}
\fint_{B_r(\widehat o)}\left|f\circ\pi\right|=\frac{1}{\nu_{\hat{g}}(B_r(\widehat o))} \sum_{\gamma\in G(r)}\int_{\cD\cap B_r(\widehat o)} \left|f\circ\pi\right|
&\le \frac{\#G(r)\nu_g(M)}{\nu_{\hat{g}}(B_r(\widehat o))} \fint_{M} \left|f\right|\\
&\le C(n) \fint_{M} \left|f\right|.
\end{align*}
If $r + D >  \sqrt{N} D$, then the volume doubling estimate yields (see e.g.~\cite[Subsection 2.3]{HSC})
\[
\nu_{\hat{g}}(B_{r+D}(\widehat o)) \le C(n)^{\frac{r+D}{r}}\nu_{\hat{g}}(B_{r}(\widehat o))\le C(n)\nu_{\hat{g}}(B_{r}(\widehat o)),
\]
and we can conclude in the same way as above.

\endproof

\section{Proof of Theorem A}

In this section, we prove Theorem \ref{Thm:A}.  All the way through we consider $\eps\in (0,1)$ and
 \begin{equation*}
\upeta:=\frac{\eps}{640 n^2}\quad\text{ and }\quad R:=64n^{2} D.
\end{equation*}
Let $(M^n,g)$ be a closed Riemannian manifold  of diameter $D$ such that 
$$b_1(M)= n \quad \text{ and } \quad \mbox{k}_{D^2}(M^n,g)\le \updelta 
$$
for some $\updelta\in (0,1/(16n))$. We consider the covering $\Gamma \longrightarrow \widehat{M} \stackrel{\pi}{\longrightarrow} M$ built in the previous section, and associated Albanese maps $\widehat{\cA}$ and $\cA$. Then the following holds.

\begin{claim}
There exists $\updelta_0(\eps,n)>0$ such that if $\updelta \le \updelta_0(\eps,n)$, then $\widehat{\cA}$  satisfies the following.
\begin{enumerate}
\item[$(a)$] $\widehat \cA$ is $(1+\upeta)$-Lipschitz.
\item[$(b)$] $\widehat \cA\colon B_{R}(\widehat o)\rightarrow \bB_R^n$ is an $\upeta R$-almost isometry.
\item[$(c)$] \label{(iii)} $ \bB_{(1-\upeta)R}^n\subset \widehat \cA\left( B_{R}(\widehat o)\right).$
\end{enumerate}
\end{claim}
This is a consequence of Proposition \ref{propAlbanese}, Theorem \ref{thm:presEucl} and i) in Remark \ref{cover}. The last assertion may be proven with degree theory, see \cite{ChPisa}, \cite[Proof of 3.2 and 3.3]{Gallot}, \cite[Proof of theorem 7.2]{CMT}. 

From now on, we assume that
\[
\updelta \le \updelta_0(\eps,n).
\]

\hfill

\textbf{Step 1. } Proceeding like in \cite{Col97,Gallot}, we construct a normal subgroup $\Gamma_0$ of $\Gamma$ with finite index, such that $\widehat{\cA}$ induces a map
\[
\cA_0\colon M_0 := \widehat M/\Gamma_0 \rightarrow \setR^n/\rho(\Gamma_0).
\]

Let $(e_1,\dots, e_n)$ be the canonical basis of $\R^n.$ Since $4 \sqrt{n}D \le (1-\upeta)R$, it follows from $(c)$ that for any $i \in \{1,\dots,n\}$, there exists $x_i \in B_R(\widehat o)$ such that
$$4\sqrt{n} De_i=\widehat \cA(x_i).$$ 
Moreover, for any $i \in \{1,\dots,n\}$, there exists $\gamma_i \in \Gamma$ such that $$\dist_{\widehat{g}}\left(\gamma_i.\widehat o,x_i\right)\le D.$$
Then we set
\[
\Gamma_0:=\langle \gamma_1,\dots,\gamma_n\rangle \subset \Gamma
\]
and
\[
\bGamma_0:= \rho(\Gamma_0) = \bigoplus_{i=1}^n \rho(\gamma_i)\Z.
\]
Let us show that
\begin{equation}\label{eq:basis}
\text{$\rho(\gamma_1),\dots,\rho(\gamma_n)$  form a basis of $\R^n$.}
\end{equation}
We know by $(a)$ that  the map $\widehat \cA$ is $(1+\upeta)$-Lipschitz.  Then for any $i \in \{1,\ldots,n\}$,  since the equivariance \eqref{eq:equivariance} of $\widehat \cA$ yields $\rho(\gamma_i)=\widehat \cA(\gamma_i.\widehat o)$, we have
\begin{equation}\label{eq:tobasis} | 4\sqrt{n} De_i -\rho(\gamma_i)  | =\left|\widehat \cA(x_i)- \widehat \cA(\gamma_i.\widehat o)\right|\le (1+\upeta)\dist_{\widehat{g}}(\gamma_i.\widehat o,x_i)\le (1+\upeta) D.\end{equation}
Then for any $\xi=\sum_{i=1}^n \xi_i e_i \in \R^n$, 
\begin{align*}
\left| 4\sqrt{n}  D \xi - \sum_{i=1}^n \xi_i\rho(\gamma_i) \right|  & = 
\left|\sum_{i=1}^n \xi_i( 4\sqrt{n}  De_i-\rho(\gamma_i))\right| \\
& \le (1+\upeta)D\sum_{i=1}^n |\xi_i| \le 2 D \sqrt{n}|\xi| ,
 \end{align*}
 so that \begin{equation}\label{eq:norm_xi}
 |  \xi | \le \frac{1}{2D\sqrt{n}}\left|\sum_{i=1}^n \xi_i\rho(\gamma_i)  \right|.
\end{equation}
Hence we get \eqref{eq:basis}. This implies that the quotient $\R^n/\bGamma_0$ is a torus $\bT^n$ which we equip with the natural flat quotient Riemannian metric $g_{\R^n/\bGamma_0}$. We also equip $M_0$ with the quotient Riemannian metric $g_0$ induced by $\widehat{g}$.

\hfill

\textbf{Step 2. } We establish the following diameter bound on $(M_0,g_0)$:
\begin{equation}\label{diam_bound}
\mathrm{diam}(M_0)\le 4(n+\sqrt{n})D.
\end{equation}

To this aim,  let us prove an intermediary result: if $\gamma \in \Gamma_0$ is such that
\begin{equation}\label{aI_assum}
|\rho(\gamma)|\le \frac{R}{2\sqrt{n}+1},
\end{equation}
then
\begin{equation}\label{aI}
\left| \dist_{\widehat g}(\widehat o,\gamma.\widehat o)-|\rho(\gamma)|\, \right|\le \upeta R.
\end{equation}
Write $\gamma \in \Gamma_0$ as 
$$\gamma=\gamma_1^{k_1}\dots\gamma_n^{k_n}$$
for some $k_1,\ldots,k_n \in \Z$.  Consider $i \in \{1,\ldots,n\}$. From $(b)$, we know that $|\dist_{\widehat g}(x_i,\widehat o)- |\widehat \cA(x_i) - \widehat \cA(\widehat o) | \, | \le \upeta R$. Since $\widehat \cA(\widehat o)=0$ and $\widehat \cA(x_i) =4 \sqrt{n}De_i$, this implies
\[
\dist_{\widehat g}(x_i,\widehat o) \le \upeta R + 4 \sqrt{n}D \le (1+4\sqrt{n})D.
\]
Then
$$\dist_{\widehat g}(\widehat o,\gamma.\widehat o)\le \sum_i|k_i|\, \max_i \dist_{\widehat g}(\widehat o,\gamma_i.\widehat o) \le (2+4\sqrt{n})D\sum_i|k_i|.$$
Since $\rho(\gamma)=\sum_{i}k_i \rho(\gamma_i),$  it follows from \eqref{eq:norm_xi} that
$$\sum_i|k_i|\le \sqrt{n}\left( \sum_i k_i^2\right)^{1/2}\le \frac{| \rho(\gamma)|}{2D} \, \cdot$$
Then we get $\dist_{\widehat g}(\widehat o,\gamma.\widehat o)\le (2\sqrt{n} +1)| \rho(\gamma)|$, so that \eqref{aI_assum} implies 
\begin{equation}
\gamma.\widehat o \in B_R(\widehat{o})
\end{equation}
and the conclusion \eqref{aI} follows from $(b)$.

We are now in a position to prove the diameter bound \eqref{diam_bound}.  Introduce the Dirichlet domain
$$\mathcal{D}_0 \coloneqq \left\{x\in \widehat M : \dist_{\widehat g}(x,\widehat o)\le  \dist_{\widehat g}(\gamma.x,\widehat o) \,\,\, \text{for all $\gamma\in \Gamma_0\setminus\{1\}$}\right\}.$$
We are going to show that 
\begin{equation}\label{eq:empty}
\mathcal{D}_0  \cap \left\{x\in \widehat M : \dist_{\widehat g}(x,\widehat o)= 2(n+\sqrt{n})D \right\} = \emptyset;
\end{equation}
then the connectedness of $\mathcal{D}_0$ will imply $\mathcal{D}_0 \subset B_{ 2 (n+\sqrt{n})D}(\widehat o)$ and \eqref{diam_bound} will be established.

The set $\cF_0 \df \sum_{i=1}^n \left[-\frac12,\frac12\right)\rho(\gamma_i)$ is a fundamental domain for the action of $\bGamma_0$ on $\R^n$; it is included in the Euclidean ball centered at the origin with radius
$$\frac{\sqrt{n}}{2}\max_i |\rho(\gamma_i)|. $$
By \eqref{eq:tobasis}, for any $i$,
$$|\rho(\gamma_i)| \le |4 \sqrt{n} D e_i|+ (1+\upeta)D \le  \left( 4 \sqrt{n}+2\right)D,$$so that
$$\frac{\sqrt{n}}{2}\max_i |\rho(\gamma_i)| \le (2n+\sqrt{n})D.$$
For any $x\in \widehat M$ there exists $\gamma_0 \in \Gamma_0$ such that $\widehat{\mathcal A}(\gamma_0.x) \in \cF_0$.  By the equivariance \eqref{eq:equivariance} of $\widehat \cA$ and the previous inequality, we get that
\begin{equation}\label{eq:control}
\left| \widehat\cA(x)+\rho(\gamma_0)\right| = \left| \widehat\cA(\gamma_0.x)\right|\le(2n+\sqrt{n})D.
\end{equation}

Now assume that $\dist_{\widehat g}(x,\widehat o)= 2(n+\sqrt{n})D$. We are going to show that $x \notin \mathcal{D}_0$. Since $2(n+\sqrt{n})D\le R$, from $(b)$ we get
$$\left| \widehat\cA(x)\right| = \left| \widehat\cA(x)- \widehat\cA(\widehat o)\right| \le  2(n+\sqrt{n})D +\upeta R.$$ Consequently,
\[
|\rho(\gamma_0)| \le\left|\widehat{\cA}(x) + \rho(\gamma_0)\right|  + \left|\widehat{\cA}(x)\right| \le (4n+3\sqrt{n})D +\upeta R.
\]
By our choices of $\upeta$ and $R$ we have
\[
(4n+3\sqrt{n})D +\upeta R \le \frac{R}{2 \sqrt{n}+1} \, \cdot
\]
Then we are in a position to apply \eqref{aI}. We get
$$\dist_{\widehat g}(\gamma_0.\widehat o,\widehat o)\le (4n+3\sqrt{n})D + 2\upeta R$$
and then
\begin{align*}
\dist_{\widehat{g}}(\gamma_0.x,\widehat o) & \le \dist_{\widehat{g}}(\gamma_0.x,\gamma_0.\widehat o)+\dist_{\widehat{g}}(\gamma_0.\widehat o,\widehat o) \le 2(n+\sqrt{n})D+(4n+3\sqrt{n})D + 2\upeta R.
\end{align*}
Since
\[
2(n+\sqrt{n})D+(4n+3\sqrt{n})D + 2\upeta R \le R,
\]
we can use $(b)$ and \eqref{eq:control} to deduce that
\[\dist_{\widehat g}(\gamma_0.x,\widehat o)\le \left| \widehat\cA(\gamma_0.x)\right|+\upeta R \le (2n+\sqrt{n})D+\upeta R  <2(\sqrt{n}+2n)D=\dist_{\widehat g}(x,\widehat o).\]
Thus $x \notin \mathcal{D}_0$ and \eqref{eq:empty} is proven.

\hfill

\textbf{Step 3.} We prove that $\cA_0:M_0 \to \setR^n/\bGamma_0$ is a $3 \upeta R$-almost isometry.  From Proposition \ref{propAlbanese}, we know that $\widehat \cA$ is surjective,  hence $\cA_0$ is surjective too. Thus we are left with proving the distance estimate. Let us introduce the following intermediate projection maps $\pi_0$ and $p_0$: 
$$\xymatrix{
{\widehat M}\ar[d]^{\pi_0}\ar[r]^{\widehat \cA} &{ \  \ \R^{n}}\ar[d]^{p_0}\\
{M_0}\ar[r]^{\cA_0 }&{\, \R^{b_1}\!/\bGamma_0} .
} $$

 Let $x,y\in M_0$.  Since $\mathcal{D}_0 \subset B_{ 2(n+\sqrt{n})D}(\widehat o)$, we can choose $\widehat x\in B_{ 2(n+\sqrt{n})D}(\widehat o)$ such that $\pi_0(\widehat x)=x$.  Let $c\colon [0,1]\rightarrow M_0$  be a  minimizing geodesic joining $x$ and $y$.  Let $\widehat c\colon [0,1]\rightarrow \widehat M$ be such that $\pi_0\circ \widehat c=c$ and $\widehat x=\widehat c(0).$ By the diameter bound \eqref{diam_bound}, we know that $\widehat y:=\widehat c(1)$ belongs to $B_{6(n+\sqrt{n})D}(\widehat o)\subset  B_{R}(\widehat o)$. Moreover, $\pi_0(\widehat y)=y$. Thus
\begin{align}\label{GHun}
\dist_{g_{\R^n/\bGamma_0}}(\cA_0(x),\cA_0(y))=\dist_{g_{\R^n/\bGamma_0}}\left(p_0(\widehat\cA(\widehat x)),p_0(\widehat\cA(\widehat y))\right) & \le \left|\widehat\cA(\widehat x)-\widehat\cA(\widehat y)\right| \nonumber\\
& \le \dist_{g_0}(x,y)+\upeta R\end{align}
thanks to $(b)$.

It remains to prove that
\[
\dist_{g_0}(x,y) - \dist_{g_{\R^n/\bGamma_0}}(\cA_0(x),\cA_0(y))  \le  3 \upeta R.
\]
We start by showing that if $\cA_0(x)=\cA_0(y)$ then $\dist_{g_0}(x,y) \le \upeta R$.  Since $\cA_0 \circ  c = \cA_0 \circ\pi_0 \circ \widehat{c} = p_0 \circ \widehat{\cA} \circ \widehat{c},$  the curve $\widehat{\cA} \circ \widehat{c}$ is a lift of the curve $\cA_0 \circ c$ joining $\vec v:=\widehat \cA(\widehat x)\in \R^n$ to $\vec w:=\widehat \cA(\widehat y)$.  Moreover,  the length of $\widehat{c}$ is less than $4(n+\sqrt{n})D$,  hence $(b)$ implies that the length of $\widehat{\cA} \circ \widehat{c}$ is less than $4(1+\upeta)(n+\sqrt{n})D$. Since $p_0(\vec v)=p_0(\vec w)$ there exists $\gamma\in \Gamma_0$ such that
\[
\widehat \cA(\widehat y)= \vec w=\vec v+\rho(\gamma)=\widehat \cA(\widehat x)+\rho(\gamma)
 \]
and 
$$| \rho(\gamma)|\le 4(1+\upeta)(n+\sqrt{n})D\le 8 (n+\sqrt{n})D.$$
As a consequence, notice that $\gamma^{-1}.\widehat{y}$ satisfies 
\begin{align*}
\dist_{\widehat{g}}(\gamma^{-1}.\widehat y,\widehat o) \le \dist_{\widehat{g}}(\widehat o,\gamma.\widehat o)+\dist_{\widehat{g}}(\widehat y,\widehat o)& \le \upeta R+ |\rho(\gamma)|+6(n+\sqrt{n}) D\\ & \le 14(\sqrt{n}+n)D+\upeta R,
\end{align*}
where we used that $\widehat{\cA}$ is an $\upeta R$-almost isometry in the second inequality. Since $14(\sqrt{n}+n) D +\upeta R \le R$, we get that $\gamma^{-1}y \in B_R(\widehat{o})$. Moreover, $\widehat \cA(\widehat y)=\widehat \cA(\widehat x)+\rho(\gamma)$, thus by the invariance of $\widehat \cA$ we have $\widehat \cA(\gamma^{-1}.\widehat y)=\widehat \cA(\widehat x)$. Then we can apply $(b)$ and obtain
\[
\dist_{g_0}(x,y) \le \dist_{\widehat g}(\gamma^{-1} . \widehat{y},\widehat{x}) \le \upeta R.
\]

Now assume that $v\coloneqq \cA_0(x)$ and $w\coloneqq \cA_0(y)$ are distinct.  We can choose $\vec v$ and $\vec w$ such that $p_0(\vec v)=v$,  $p_0(\vec w)=w$,
$$\dist_{g_{\R^n/\bGamma_0}}(v,w)=|\vec v-\vec w| \qquad \text{and} \qquad \vec v,\vec w\in \bB_{2(\sqrt{n}+n)D}^n.$$
Since $2(\sqrt{n}+n)D<(1-\upeta)R$, from $(c)$ we know that there exist $\widehat x',\widehat y'\in B_R(\widehat o)$ such that
$\widehat \cA(\widehat x')=\vec v$ and $\widehat \cA(\widehat y')=\vec w$.
Then $\cA_0(\pi_0(\widehat x'))=p_0(\widehat\cA(\widehat x'))=v=\cA_0(x)$ and $\cA_0(\pi_0(\widehat y'))=p_0(\widehat\cA(\widehat y'))=w=\cA_0(y)$, thus
\begin{align*}
\dist_{g_0}(x,y) & \le 2 \upeta R+\dist_{g_0}(\pi_0(\widehat x'),\pi_0(\widehat y')) \quad \quad \, \, \,\,  \text{by the previous paragraph}\\
& \le 2 \upeta R+\dist_{\widehat g}(\widehat x',\widehat y') \\
& \le 3 \upeta R+|\vec v-\vec w| \qquad \qquad  \quad \text{by $(b)$}\\
& =3 \upeta R+\dist_{g_{\R^n/\bGamma_0}}(\cA_0(x),\cA_0(y)).
\end{align*}

Thus we have shown that $\cA_0\colon M_0\rightarrow \R^n/\bGamma_0$ is a $3\upeta R$-almost isometry. 

\hfill

\textbf{Step 4.} We conclude. Repeating the arguments of Step 3  with the commutative diagram
$$\xymatrix{
{M_0}\ar[d]^{}\ar[r]^{\cA_0} &{ \, \R^{n}\!/\bGamma_0}\ar[d]^{}\\
{M}\ar[r]^{\cA_0 }&{\, \R^{n}\!/\bGamma} .
} $$

we get that $\cA\colon M\rightarrow\R^n/\bGamma$ is a $9\upeta R$-almost isometry. Since $9 \upeta R=9\eps D/10$, this concludes the proof of Theorem \ref{Thm:A}.

\section{Proof of Theorem \ref{Thm:B}}
In this section, we prove Theorem \ref{Thm:B}. Let $f\colon [0,1]\rightarrow \R_+$ be a non-decreasing function satisfying \eqref{eq:SK}.  Let $(M^n,g)$ be a closed Riemannian manifold of diameter $D$ such that
\begin{itemize}
\item[•] $b_1(M)=n$,
\item[•] $\mbox{k}_{D^2}(M^n,g)\le \updelta$ for some $\updelta \in (0,1/(16n))$,
\item[•] $\mbox{k}_{tD^2}(M^n,g)\le f(t)$ for any $t \in (0,1]$.
\end{itemize}
We consider the covering $\Gamma \longrightarrow \widehat{M} \stackrel{\pi}{\longrightarrow} M$ built in Section 4, and associated Albanese maps $\widehat{\cA}$ and $\cA$.  Set $\widehat{g}\coloneqq\pi^*g$.  From \eqref{eq:katochapeau}, we get that
\begin{itemize}
\item $\mbox{k}_{(2D)^2}(\widehat{M}^n,\widehat{g})  \le 4 \updelta$,
\item $\mbox{k}_{t(2D)^2}(\widehat{M}^n,\widehat{g})  \le 4f(t)$ for any $t \in (0,1]$.
\end{itemize}
Let $\upbeta(n,4f)$ be given by Proposition  \ref{Reife}. Set
\[
\upeta \coloneqq \upbeta(n,4f)/4.
\]
Let $\widehat{H}$ be the heat kernel of $\widehat{M}$, and $\widehat{\uptheta}(t,x)\coloneqq (4 \pi t)^{n/2}\widehat{H}(t,x,x)$ for any $(t,x)\in \setR_+ \times \widehat{M}$.  Then the following holds.

\begin{claim}\label{claimThB} 
There exists $\updelta_0(n,f)  \in (0,\upeta]$ such that if $\updelta \le \updelta_0(n,f)$,  then 
\begin{equation}\label{estitheta}
\widehat{\uptheta}((2D)^2,\widehat o)\le 1+\upeta
\end{equation}
and
\begin{equation}\label{estieta}
 \left((2D)^2\fint_{B_{2D}(\widehat o)} |\nabla d\widehat\cA|^2\di\nu_{\widehat g}\right)^{\frac 12}+\fint_{B_{2D}(\widehat o)}\left| d\widehat\cA\,{}^td\widehat\cA-\Id_n\right| \di\nu_{\widehat g}\le \upeta.
\end{equation}
\end{claim}
This claim puts us in a position to apply Proposition \ref{Reife}: we get that the map $\widehat\cA\colon B_{3D/2}(\widehat o)\rightarrow \R^n$ is  a diffeomorphism onto its image.  Therefore, the Albanese map $\cA:M \to \R^n/\bGamma$ is a local diffeomorphism,  hence it is a finite cover. Since a torus is finitely covered by tori only, we get that $M$ is diffeomorphic to a torus.  Then the Albanese map is necessarily a diffeomorphism, by construction. As a consequence, the conclusion of Theorem \ref{Thm:B} holds.

Let us now prove Claim \ref{claimThB}. From Proposition \ref{propAlbanese}, we know that if we choose $\updelta_0(n,f)\le \min(\updelta(n),2^{-6}, (C(n)^{-1} \upeta)^3)$, then \eqref{estieta} holds.  Let us prove that we can also choose $\updelta_0(n,f)$ such that \eqref{estitheta} holds.  To this aim, we introduce the following almost Euclidean heat kernel on $\widehat{M}$: for any $\eps \in (0,1)$,  $x,y \in \widehat{M}$ and $t>0$,  we set
\begin{equation}\label{eq:almostEucl}
\mathbb{H}_\eps(t,x,y) := \frac{1}{(1+\eps)(4\pi t)^{n/2}} e^{-(1+\eps)\frac{\dist_{\widehat g}^2(x,y)}{4t}}.
\end{equation}

\textbf{Step 1}. We prove the following Cheeger--Yau type estimate: for any $\eps \in (0,1)$ and any integer $\ell \ge 4$, there exists $\updelta_1(n,f,\eps,\ell)>0$ such that if
\[
\katoD\le \updelta \le \updelta_1(n,f,\eps,\ell),\]
then for any $x,y \in \widehat{M}$ and $t\in (0, \ell D^2]$,
\begin{equation}\label{eq:step1}
\mathbb{H}_\eps(t,x,y) \le \widehat{H}(t,x,y).
\end{equation}

Let us set $\Gamma(\tau)\coloneqq e^{8 \sqrt{n\mbox{k}_{\tau}(M^n,g)}}$ for any $\tau>0$.  Since $\widehat{M}$ satisfies the Dynkin condition \eqref{Dynkin}, it satisfies the Li-Yau inequality from \cite[Proposition 3.3]{C16}. Then we can proceed as in the proof of \cite[Proposition 2.12]{CMT} to get that for any $s \in (0,t)$, any positive solution $u$ of the heat equation on $\widehat M\times [0,\ell D^2]$ satisfies
\[
\log\left(\frac{u(s,x)}{u(t,y)}\right) \le  \left( \frac{t}{s} \right)^{n/2} e^{\Gamma(t)\frac{\dist_{\widehat{g}}^2(x,y)}{4(t-s)}}e^{\frac{n}{2}\int_s^t\frac{\Gamma(\tau)-1}{\tau} \di \tau}\,\cdot
\]
Apply this inequality with $u(\cdot,\cdot)=\widehat H (\cdot, x, \cdot)$ to get
\begin{align*}
\frac{(4 \pi s)^{n/2}\widehat H (s, x, x)}{(4 \pi t)^{n/2}\widehat H (t, x, y)} & \le e^{\Gamma(t)\frac{\dist_{\widehat{g}}^2(x,y)}{4(t-s)}}e^{\frac{n}{2}\int_s^t\frac{\Gamma(\tau)-1}{\tau} \di \tau}  \\
& \le e^{\Gamma(\ell D^2)\frac{\dist_{\widehat{g}}^2(x,y)}{4(t-s)}}e^{C(n)\int_0^{\ell D^2}  \sqrt{\mbox{\small k}_{\tau}(M^n,g)} \frac{\di \tau}{\tau}} \, \cdot
\end{align*}
Letting $s \downarrow 0$ yields
\[
\frac{e^{-\Gamma(\ell D^2)\frac{\dist_{\widehat{g}}^2(x,y)}{4t}}}{(4 \pi t)^{n/2}} e^{-C(n)\int_0^{\ell D^2}  \sqrt{\mbox{\small k}_{\tau}(M^n,g)} \frac{\di \tau}{\tau}} \le \widehat H (t, x, y).
\]
Since $\mbox{k}_{\ell D^2}(M^n,g)\le \ell \mbox{k}_{D^2}(M^n,g) \le \ell \updelta$, we know that there exists $\updelta_2(n,f,\eps,\ell)>0$ such that if $\updelta \le \updelta_2(n,f,\eps,\ell)$, then
\[
\Gamma(\ell D^2) \le 1+ \eps.
\]
To conclude, let us show that there exists $\updelta_1(n,f,\eps,\ell)\le \updelta_2(n,f,\eps,\ell)$ such that if $\updelta \le \updelta_1(n,f,\eps,\ell)$, then
\begin{equation}\label{eq:tofind}
e^{C(n)\int_0^{\ell D^2}  \sqrt{n \mbox{\small k}_{\tau}(M^n,g)} \frac{\di \tau}{\tau}} \le 1+\eps.
\end{equation}
Since $\mbox{k}_{\tau}(M^n,g)\le f(\tau/D^2)$ for any $\tau \in (0,D^2)$,  for any $\sigma \in (0,1)$ we have
\begin{align*}\int_0^{\ell D^2} \sqrt{\mbox{k}_{\tau}(M^n,g)}\frac{\di \tau}{\tau}&\le  \int_0^{\sigma D^2} \sqrt{\mbox{k}_{\tau}(M^n,g)} \frac{\di \tau}{\tau} +\log(\ell/\sigma)\sqrt{\mbox{k}_{\ell D^2}(M^n,g)}  \\
& \le \int_0^\sigma \sqrt{f(\tau)}\frac{\di \tau}{\tau}+\log(\ell/\sigma) \sqrt{\ell \mbox{k}_{D^2}(M^n,g)}.
\end{align*}
Therefore, to get \eqref{eq:tofind},  first we choose $\sigma(f) \in (0,1)$ such that 
$$ e^{ C(n) \int_0^{\sigma(f)} \frac{\sqrt{f(s)}}{s} \di s } \le \sqrt{1+\eps},$$ then we choose $\updelta_1(n,f,\eps,\ell) \le \updelta_2(n,f,\eps,\ell)$ such that 
$$e^{C(n)\log(\ell/\sigma(f))\sqrt{\ell \, \updelta_1(n,f,\eps,\ell)}}\le  \sqrt{1+\eps}.$$

\textbf{Step 2.} We prove that for any $\eps \in (0,1)$ and $t>D^2$,
\begin{equation}\label{eq:step2}
\int_{\widehat{M}}\mathbb{H}_\eps(t,\widehat{o},y) \di \nu_{\widehat{g}} (y) \ge \frac{1-C(n) \sqrt[3]{\updelta}}{\left(1+\eps\right)^{\frac n2 +1}}\left(1-C(n)\left(\frac{D}{\sqrt{t}}\right)^{n+2}-C(n)e^{-\frac{D^2}{5\sqrt[3]{\updelta}t}} \right).
\end{equation}
Set $s:=t/(1+\eps)$. Then
$$\int_{\widehat M} \bH_\eps(t,\widehat o,y)\di\nu_{\widehat g}(y)=\frac{1}{(1+\eps)^{1+\frac n2}}\int_{\widehat M} \frac{1}{(4\pi s)^{\frac n2}}e^{-\frac{\dist_{\widehat{g}}^2(\widehat o,y)}{4s}}\di\nu_{\widehat g}(y).$$
By Cavalieri's principle,
\begin{align*}\int_{\widehat M} \frac{1}{(4\pi s)^{\frac n2}}e^{-\frac{\dist_{\widehat{g}}^2(\widehat o,y)}{4s}}\di\nu_{\widehat g}(y)&=\int_0^{+\infty}\frac{e^{-\frac{r^2}{4s}}}{(4\pi s)^{\frac n2}}\frac{r}{2s}\nu_{\widehat g}\left(B_r(\widehat o)\right) \di r \\
&\ge \int_{D}^{D\updelta^{-\frac16}}\frac{e^{-\frac{r^2}{4s}}}{(4\pi s)^{\frac n2}}\frac{r}{2s}\nu_{\widehat g}\left(B_r(\widehat o)\right) \di r\\
&\ge \left(1-C(n)\sqrt[3]{\updelta}\right)\int_{D}^{D\updelta^{-\frac16}}\frac{e^{-\frac{r^2}{4s}}}{(4\pi s)^{\frac n2}}\frac{r}{2s}\omega_nr^n \di r,\end{align*}
where we use the volume lower bound \eqref{volAhatM} to get the last line.  Set
\[
\phi_{n,s}(r):=\frac{e^{-\frac{r^2}{4s}}}{(4\pi s)^{\frac n2}}\frac{r}{2s}\omega_nr^n
\]
and note that 
\[
\int_0^{+\infty}\phi_{n,s}(r) \di r= 1.
\]
Then
\begin{equation}\int_{\widehat M} \frac{1}{(4\pi s)^{\frac n2}}e^{-\frac{\dist_{\widehat{g}}^2(\widehat o,y)}{4s}}\di\nu_{\widehat g}(y)\ge \left(1-C(n)\sqrt[3]{\updelta}\right)\left(1-\int_{0}^{D}\phi_{n,s}(r) \di r - \int_{D\updelta^{-\frac16}}^{+\infty}\phi_{n,s}(r) \di r\right).\end{equation}
A direct computation shows that
\begin{equation}\int_{0}^{D}\phi_{n,s}(r) \di r\le C(n) \left(\frac{D}{\sqrt{s}}\right)^{n+2} \quad \text{and} \quad \int_{D\updelta^{-\frac16}}^{+\infty}\phi_{n,s}(r) \di r\le C(n)e^{-\frac{D^2}{5\sqrt[3]{\updelta}s}}.\end{equation}
Hence we get \eqref{eq:step2}.\\

\textbf{Step 3.} We conclude. Consider $\eps \in (0,1)$. From the proof of \cite[Lemma 5.7]{CMT2},  we know that for any integer $\ell \ge 4$,
\[
\widehat{\uptheta}((2D)^2,\widehat o)\le \widehat{\uptheta}(\ell D^2,\widehat o) e^{C(n)\sqrt{\updelta\ell}}.
\]
We are going to show that we can choose $\updelta$ small enough and $\ell$ large enough to ensure that
\begin{equation}\label{constraint2}
\widehat{\uptheta}(\ell D^2,\widehat o)\le \sqrt{1+\upeta}
\end{equation}
and
\begin{equation}\label{constraint1}
e^{C(n)\sqrt{\updelta\ell}} \le \sqrt{1+\upeta}.
\end{equation}
For the sake of brevity,  let us set $\tau: = \ell D^2$.  Assume that $\updelta \le \updelta_1(n,f,\eps,\ell)$ where $ \updelta_1(n,f,\eps,\ell)$ is given by Step 1. The semi-group law yields
\begin{align}\label{eq:I}
\widehat{\uptheta}(\tau,\widehat o)&=\left(4\pi\tau\right)^{\frac n2}\int_{\widehat{M}} {\widehat H}^2\left(\tau/2,\widehat o,y\right)\di\nu_{\widehat g}(y)\nonumber \\
&=\left(4\pi\tau\right)^{\frac n2}\int_{\widehat{M}} \left({\widehat H}^2\left(\tau/2,\widehat o,y\right)-\bH^2_\eps\left(\tau/2,\widehat o,y\right)\right)\di\nu_{\widehat g}(y)\nonumber \\
&\qquad\qquad\qquad\qquad\qquad\qquad+\frac{1}{1+\eps}\int_{\widehat{M}}\bH_\eps(\tau/4,\widehat o,y)\di\nu_{\widehat g}(y) \nonumber \\
& \fd I + II.\end{align}
By Step 1 and stochastic completeness,
\begin{equation}\label{eq:III}
II \le \frac{1}{1+\eps} \le 1.
\end{equation}
By Step 1 we also  know that $0 \le \widehat{H}^2 - \bH_\eps^2 = (\widehat{H} - \bH_\eps)(\widehat{H} + \bH_\eps) \le (\widehat{H} - \bH_\eps)2 \widehat{H}$. Moreover, the heat kernel upper bound from Proposition \ref{Prop:heatK} and the volume lower bound \eqref{volAhatM} imply that if
\begin{equation}\label{constraint3}
\ell\updelta^{\frac 13}\le 2,
\end{equation}
then
$$\widehat H\left(\tau/2,\widehat o,y\right)\le \frac{C(n)}{\tau^{\frac n2}} \, ,$$
so that
\begin{align}\label{eq:II}
I &\le \left(4\pi\tau\right)^{\frac n2}\!\!\int_{\widehat  M}  \left(\widehat {H}\left(\tau/2,\widehat o,y\right)-\bH_\eps\left(\tau/2,\widehat o,y\right)\right) 2 \widehat {H} \left(\tau/2,\widehat o,y\right)\di\nu_{\widehat g}(y)\nonumber \\
& \le C(n) \int_{\widehat  M}  \left(\widehat {H}\left(\tau/2,\widehat o,y\right)-\bH_\eps\left(\tau/2,\widehat o,y\right)\right)\di\nu_{\widehat g}(y) \nonumber \\
& =  C(n)\left(1-\int_{\widehat  M}\bH_\eps\left(\tau/2,\widehat o,y\right)\di\nu_{\widehat g}(y)\right),
\end{align}
by stochastic completeness.  By combining \eqref{eq:I}, \eqref{eq:III}, \eqref{eq:II}, and thanks to Step 2, we eventually get
\begin{align*}
\widehat{\uptheta}(\tau,\widehat o)&\le 1+C(n)\left(1-\frac{1-C(n)\sqrt[3]{\updelta}}{\left(1+\eps\right)^{\frac n2 +1}}\left(1-C(n)\left(\frac{1}{\sqrt{\ell}}\right)^{n+2}-C(n)e^{-\frac{1}{5\sqrt[3]{\updelta}\ell}}\right)\right)\\
&\le 1+C(n)\left(\eps+\sqrt[3]{\updelta}+\left(\frac{1}{\sqrt{\ell}}\right)^{n+2}+e^{-\frac{1}{5\sqrt[3]{\updelta}\ell}}\right).\end{align*}
Then we choose successively:
\begin{enumerate}
\item $\eps\in (0,1)$ such that $C(n)\eps\le \frac{\upeta}{12},$
\item $\ell\ge 4$ such that $C(n)\left(\frac{1}{\sqrt{\ell}}\right)^{n+2}\le \frac{\upeta}{12},$
\item $\updelta\le \updelta_1(n,f,\eps,\ell)$ such that \eqref{constraint1} and \eqref{constraint3} hold together with
$$C(n)\sqrt[3]{\updelta}\le \frac{\upeta}{12}\qquad \text{and} \qquad C(n)e^{-\frac{1}{5\sqrt[3]{\updelta}\ell}}\le \frac{\upeta}{12} \, \cdot $$
\end{enumerate}
This implies $\widehat{\uptheta}(\ell D^2,\widehat o) \le 1+ \upeta/3 \le \sqrt{1+\upeta}$ and concludes the proof.

\section{Appendix : Almost surjectivity}

In this appendix, we point out that almost splitting maps are almost surjective, without any assumption on the Ricci curvature.  We single out this fact from the proof of \cite[Theorem 1.2]{cheeger_structure_2000} (see also \cite[Section 2]{CheegerColdingTian} for variants).

\begin{theorem}\label{thm=presquesurjective} Let $(M^n,g)$ be a complete Riemannian manifold and $k\in\{1,\ldots,n\}$. There exist $\upeta(n,k)\in (0,1)$ and $C(n,k)>0$ such that for any $o \in M$ and $r>0$, if there exists $\Phi\colon B_r(o) \rightarrow \R^k$ smooth and $\eps\in (0,\upeta(n,k))$ such that
\begin{enumerate}[i)]
\item $\Phi(o)=0$,
\item $\| d\Phi \|_{L^\infty(B_r(o))}\le 1+\eps$,
\item $\fint_{B_r(o)} \left|d\Phi ^t d\Phi-\Id_k\right|\di\nu_g\le \eps$,
\item $r\fint_{B_r(o)} \left|\nabla d\Phi\right|\di\nu_g\le \eps$,
\end{enumerate}
then
$$\cH^k\left(\bB^k_r\setminus \Phi(B_r(o))\right)\le C(n,k)r^k\ \eps.$$
\end{theorem}
\proof By scaling,  there is no loss of generality in assuming $r=1$, what we do from now on. Set $B:=B_1(o)$ and 
$$w:=\sqrt{\det d\Phi ^t d\Phi}=\left|d\Phi_1\wedge\dots \wedge d\Phi_k\right|, $$ and recall that the coaera formula gives that for any $f\in L^1(B)$,
\begin{equation}\label{eq:coarea}\int_{B} f\di\nu_g=\int_{\R^k}\left(\int_{\Phi^{-1}(z)}\frac{f}{w}\di\cH^{n-k}\right)\di z.\end{equation}
Acting as in \cite{cheeger_structure_2000}, we introduce a function $J: \setR^k \to \setR_+$ which provides a weighted measure of the fibers $\Phi^{-1}(z)$.  Let $\chi\colon \R_+\rightarrow [0,1]$ be a smooth function such that $\chi=0$ on $[0,1/4]$ and $\chi=1$ on $[1/2,+\infty)$.  For any $z \in \setR^k$, set
$$J(z):=\int_{\Phi^{-1}(z)}\chi \circ w^2 \di\cH^{n-k}.$$
Note that if $z \notin \Phi(B)$, then $J(z)=0$.  Moreover, the presence of $\chi$ in the integrand ensures that the integral may be taken over $$\Sigma_z:=\Phi^{-1}(z) \cap \{w>0\}$$ which is a smooth $(n-k)$-dimensional submanifold of $B$. By \textit{ii)}, we know that $\Phi(B) \subset \mathbb{B}^k_{1+\eps}$. Therefore, by the Poincaré inequality,
\begin{equation}\label{eq:byPoincaré}\cH^k\left(\bB^k_{1+\eps}\setminus \Phi(B)\right) \bar J  \le \int_{\bB^k_{1+\eps}} \left|J(z)-\bar J\right| \di z \le C(n,k)\int_{\bB^k_{1+\eps}}|\nabla J(z)| \di z\end{equation}
where $$\bar J:=\fint_{\bB^k_{1+\eps}} J.$$
Let us estimate $|\nabla J(z)|$.  For any $v \in \setR^k$ and $x \in \{w>0\}$,  we let $X_v(x)$ be the unique element in $T_x\Sigma_{\Phi(x)}^{\perp}$ such that $d_x \Phi(X_v(x))=v$. Then there exists $\xi\colon B\rightarrow \R^k$ such that $\sum_{\alpha=1}^k\xi_\alpha \nabla \Phi_{\alpha}=X_v$. We easily get that $$|X_v|\le \frac{2^{k-1}}{w} |v|. $$
Moreover
\begin{equation}\label{GradJ}
\nabla_vJ(z)=\int_{\Sigma_z}(\chi' \circ w^2)\nabla_{X_v}w^2\di\cH^{n-k}+\int_{\Sigma_z}(\chi \circ w^2)\langle \vec H,X_v\rangle\di\cH^{n-k}\end{equation}
where $\vec H$ is the mean curvature vector of $\Sigma_z$. We easily compute that
$$\langle \vec H,X_v\rangle=-\sum_{\alpha=1}^k\sum_{i=1}^{n-k}\xi_\alpha\nabla d\Phi_\alpha(e_i,e_i)$$ where
$(e_1,\dots,e_{n-k})$ is an orthonormal basis of $T_x\Sigma_z$.
Hence
$$Â \left|\langle \vec H,X_v\rangle\right|\le C(n,k)\,  \frac{|\nabla d\Phi|}{w}\, |v|.$$
We also have
$$|\nabla_{X_v}w^2|\le C(n)|v| |\nabla d\Phi|.$$
Using the fact that integration in \eqref{GradJ} is done only on the set  $\{w\ge 1/2\}$, we get that 
 $$|\nabla J(z)|\le C(n,k)\int_{\Phi^{-1}(z)} \frac{|\nabla d\Phi|}{w}\di\cH^{n-k}.$$ Then the coarea formula \eqref{eq:coarea} applied to $|\nabla d \phi|$ together with \textit{iv)} yields that
\begin{equation}\label{eq:bycoarea}
\int_{\R^k}|\nabla J(z)|\di z \le C(n,k)\eps\, \nu_g (B).\end{equation}
Now we bound $\bar J$ from below. The coarea formula gives that
$$\bar J=\frac{1}{\omega_k} \int_{B} (\chi \circ w^2) \, w\di\nu_g\ge \frac{1}{\sqrt{2}\, \omega_k}\nu_g(\{w^2> 1/2\}).$$
But there is a constant $C(n)>0$ such that $|w^2-1|\le C(n) \left| d\Phi{}^td\Phi-\Id_k\right|$, so that
\begin{align*}\nu_g(\{w^2< 1/2\})& \le \nu_g \left(\left\{ \left|d\Phi{}^td\Phi-\Id_k\right|>1/(2C(n))\right\}\right)\\
& \le 2C(n) \int_{B}\left|d\Phi{}^td\Phi-\Id_k\right|\di\nu_g\le 2C(n)\nu_g( B) \eps.\end{align*}
As a consequence, if 
$ 4C(n)\eps<1$  then $\nu_g(\{w^2\ge 1/2\})\ge \frac12 \nu_g(B)$, so that
\begin{equation}\label{eq:loweronJ}
\bar J \ge \frac{1}{2 \sqrt{2} \omega_k} \nu_g(B).
\end{equation}
Combining \eqref{eq:byPoincaré}, \eqref{eq:bycoarea} and \eqref{eq:loweronJ},  we get 
$$ \cH^k\left(\bB^k_{1+\eps}\setminus \Phi(B)\right)\le C(n,k)\,\eps.$$
Finally,
$$\cH^k\left(\bB_1^k\setminus \Phi(B)\right)\le \cH^k\left(\bB^k_{1+\eps}\setminus \Phi(B)\right)+\cH^k\left(\bB^k_{1+\eps}\setminus\bB^k_1)\right)\le C(n,k)\,\eps+C(k)\eps.$$
 
\endproof

\bibliographystyle{alpha} 
\bibliography{CMT_Betti.bib}

\begin{thebibliography}{MMP22}

\bibitem[BG21]{braun2021heat}
Mathias Braun and Batu G{\"u}neysu.
\newblock Heat flow regularity, {B}ismut--{E}lworthy--{L}i's derivative
  formula, and pathwise couplings on {R}iemannian manifolds with {K}ato bounded
  {R}icci curvature.
\newblock {\em Electronic Journal of Probability}, 26:1--25, 2021.

\bibitem[BNS22]{BNS}
Elia Bru{\`e}, Aaron Naber, and Daniele Semola.
\newblock Boundary regularity and stability for spaces with {R}icci bounded
  below.
\newblock {\em Invent. Math.}, 228(2):777--891, 2022.

\bibitem[BPS21]{MR4277822}
Elia Bru{\`e}, Enrico Pasqualetto, and Daniele Semola.
\newblock Rectifiability of {R}{C}{D}({K},{N}) spaces via {$\delta$}-splitting
  maps.
\newblock {\em Ann. Fenn. Math.}, 46(1):465--482, 2021.

\bibitem[Car19]{C16}
Gilles Carron.
\newblock Geometric inequalities for manifolds with {R}icci curvature in the
  {K}ato class.
\newblock {\em Ann. Inst. Fourier (Grenoble)}, 69(7):3095--3167, 2019.

\bibitem[CC97a]{ChCo97}
Jeff Cheeger and Tobias~H. Colding.
\newblock On the structure of spaces with {R}icci curvature bounded below. {I}.
\newblock {\em J. Differential Geom.}, 46(3):406--480, 1997.

\bibitem[CC97b]{CheegerColdingI}
Jeff Cheeger and Tobias~H. Colding.
\newblock On the structure of spaces with {R}icci curvature bounded below. {I}.
\newblock {\em J. Differential Geom.}, 46(3):406--480, 1997.

\bibitem[CC00]{cheeger_structure_2000}
Jeff Cheeger and Tobias~H. Colding.
\newblock On the structure of spaces with {Ricci} curvature bounded below.
  {II}.
\newblock {\em Journal of Differential Geometry}, 54(1), January 2000.

\bibitem[CCT02]{CheegerColdingTian}
Jeff Cheeger, Tobias~H Colding, and Gang Tian.
\newblock On the singularities of spaces with bounded {R}icci curvature.
\newblock {\em Geometric \& Functional Analysis GAFA}, 12(5):873--914, 2002.

\bibitem[Che01]{ChPisa}
Jeff Cheeger.
\newblock {\em Degeneration of {R}iemannian metrics under {R}icci curvature
  bounds}.
\newblock Lezioni Fermiane. [Fermi Lectures]. Scuola Normale Superiore, Pisa,
  2001.

\bibitem[Che21]{Chen:2020aa}
Lina Chen.
\newblock {Segment Inequality and Almost Rigidity Structures for Integral Ricci
  Curvature}.
\newblock {\em International Mathematics Research Notices}, 04 2021.
\newblock rnab065.

\bibitem[CJN21]{CJN}
Jeff Cheeger, Wenshuai Jiang, and Aaron Naber.
\newblock Rectifiability of singular sets of noncollapsed limit spaces with
  {R}icci curvature bounded below.
\newblock {\em Ann. of Math. (2)}, 193(2):407--538, 2021.

\bibitem[CMT21]{CMT}
Gilles Carron, Ilaria Mondello, and David Tewodrose.
\newblock Limits of manifolds with a {K}ato bound on the {R}icci curvature.
\newblock ArXiV Preprint: 2102.05940, 2021.

\bibitem[CMT22]{CMT2}
Gilles Carron, Ilaria Mondello, and David Tewodrose.
\newblock Limits of manifolds with a {K}ato bound on the {R}icci curvature.
  {I}{I}.
\newblock ArXiV Preprint: 2205.01956, 2022.

\bibitem[Col97]{Col97}
Tobias~H. Colding.
\newblock Ricci curvature and volume convergence.
\newblock {\em Ann. of Math. (2)}, 145(3):477--501, 1997.

\bibitem[CR21]{CarronRose}
Gilles Carron and Christian Rose.
\newblock Geometric and spectral estimates based on spectral {R}icci curvature
  assumptions.
\newblock {\em J. Reine Angew. Math.}, 772:121--145, 2021.

\bibitem[CT22]{CT19}
Gilles Carron and David Tewodrose.
\newblock A rigidity result for metric measure spaces with {Euclidean} heat
  kernel.
\newblock {\em Journal de l{\textquoteright}\'Ecole polytechnique {\textemdash}
  Math\'ematiques}, 9:101--154, 2022.

\bibitem[CY81]{CheegerYau}
Jeff Cheeger and Shing~Tung Yau.
\newblock A lower bound for the heat kernel.
\newblock {\em Comm. Pure Appl. Math.}, 34(4):465--480, 1981.

\bibitem[Gal83]{Gallot2}
Sylvestre Gallot.
\newblock In\'{e}galit\'{e}s isop\'{e}rim\'{e}triques, courbure de {R}icci et
  invariants g\'{e}om\'{e}triques. {II}.
\newblock {\em C. R. Acad. Sci. Paris S\'{e}r. I Math.}, 296(8):365--368, 1983.

\bibitem[Gal88]{GallotInt}
Sylvestre Gallot.
\newblock Isoperimetric inequalities based on integral norms of {R}icci
  curvature.
\newblock Number 157-158, pages 191--216. 1988.
\newblock Colloque Paul L\'{e}vy sur les Processus Stochastiques (Palaiseau,
  1987).

\bibitem[Gal98]{Gallot}
Sylvestre Gallot.
\newblock Volumes, courbure de {R}icci et convergence des vari\'{e}t\'{e}s
  (d'apr\`es {T}. {H}. {C}olding et {C}heeger-{C}olding).
\newblock Number 252, pages Exp. No. 835, 3, 7--32. 1998.
\newblock S\'{e}minaire Bourbaki. Vol. 1997/98.

\bibitem[Gig15]{GigliMAMS}
Nicola Gigli.
\newblock On the differential structure of metric measure spaces and
  applications.
\newblock {\em Mem. Amer. Math. Soc.}, 236(1113):vi+91, 2015.

\bibitem[GM88]{GallotMeyer}
Sylvestre Gallot and Daniel Meyer.
\newblock D'un r{\'e}sultat hilbertien {\`a} un principe de comparaison entre
  spectres. applications.
\newblock {\em Ann. Scient. Ec. Norm. Sup.}, 21:561--591, 1988.

\bibitem[GP15]{MR3412360}
Batu G\"{u}neysu and Diego Pallara.
\newblock Functions with bounded variation on a class of {R}iemannian manifolds
  with {R}icci curvature unbounded from below.
\newblock {\em Math. Ann.}, 363(3-4):1307--1331, 2015.

\bibitem[GR18]{MR3814057}
Nicola Gigli and Chiara Rigoni.
\newblock Recognizing the flat torus among {$\text{RCD}^*(0,N)$} spaces via the
  study of the first cohomology group.
\newblock {\em Calc. Var. Partial Differential Equations}, 57(4):Paper No. 104,
  39, 2018.

\bibitem[Gro81]{Gromov1}
Mikhael Gromov.
\newblock {\em Structures m\'{e}triques pour les vari\'{e}t\'{e}s
  riemanniennes}, volume~1 of {\em Textes Math\'{e}matiques [Mathematical
  Texts]}.
\newblock CEDIC, Paris, 1981.
\newblock Edited by J. Lafontaine and P. Pansu.

\bibitem[HSC01]{HSC}
Waldemar Hebisch and Laurent Saloff-Coste.
\newblock On the relation between elliptic and parabolic {H}arnack
  inequalities.
\newblock {\em Ann. Inst. Fourier (Grenoble)}, 51(5):1437--1481, 2001.

\bibitem[Li80]{Li}
Peter Li.
\newblock On the sobolev constant and the $ p $-spectrum of a compact
  riemannian manifold.
\newblock In {\em Annales scientifiques de l'{\'E}cole Normale Sup{\'e}rieure},
  volume~13, pages 451--468, 1980.

\bibitem[MMP22]{Mondello:2021aa}
Ilaria Mondello, Andrea Mondino, and Raquel Perales.
\newblock An upper bound on the revised first {B}etti number and a torus
  stability result for {R}{C}{D} spaces.
\newblock {\em Comment. Math. Helv.}, 97(3):555--609, 2022.

\bibitem[MW19]{MR3987869}
Andrea Mondino and Guofang Wei.
\newblock On the universal cover and the fundamental group of an {${\rm
  RCD}^*(K,N)$}-space.
\newblock {\em J. Reine Angew. Math.}, 753:211--237, 2019.

\bibitem[PW97]{PW1}
Peter Petersen and Guofang Wei.
\newblock Relative volume comparison with integral curvature bounds.
\newblock {\em Geom. Funct. Anal.}, 7(6):1031--1045, 1997.

\bibitem[PW01]{PW2}
Peter Petersen and Guofang Wei.
\newblock Analysis and geometry on manifolds with integral {R}icci curvature
  bounds. {II}.
\newblock {\em Trans. Amer. Math. Soc.}, 353(2):457--478, 2001.

\bibitem[Ros19]{Rose}
Christian Rose.
\newblock Li-{Y}au gradient estimate for compact manifolds with negative part
  of {R}icci curvature in the {K}ato class.
\newblock {\em Ann. Global Anal. Geom.}, 55(3):443--449, 2019.

\bibitem[RS17]{RoseStollmann}
Christian Rose and Peter Stollmann.
\newblock The {K}ato class on compact manifolds with integral bounds on the
  negative part of {R}icci curvature.
\newblock {\em Proc. Amer. Math. Soc.}, 145(5):2199--2210, 2017.

\bibitem[RS20]{rose2018manifolds}
Christian Rose and Peter Stollmann.
\newblock Manifolds with {R}icci curvature in the {K}ato class: heat kernel
  bounds and applications.
\newblock {\em Analysis and Geometry on Graphs and Manifolds}, 461:76, 2020.

\bibitem[Rud73]{Rudin}
Walter Rudin.
\newblock {\em Functional analysis}.
\newblock McGrawHill, New York, 1973.

\bibitem[RW20]{rosewei}
Christian Rose and Guofang Wei.
\newblock Eigenvalue estimates for {K}ato-type {R}icci curvature conditions.
\newblock ArXiV Preprint: 2003.07075, 2020.

\end{thebibliography}
\end{document}